\pdfoutput=1
\documentclass[reqno,oneside]{amsart}
\usepackage[margin=3cm]{geometry}
\usepackage{amsmath}
\usepackage{amssymb}
\usepackage{amsbsy}
\usepackage{amsfonts}
\usepackage{dsfont}
\usepackage[utf8]{inputenc}
\usepackage{enumerate}

\usepackage[ruled,vlined,norelsize]{algorithm2e}
\usepackage[dvips]{epsfig}
\usepackage{colortbl}
\usepackage{algcompatible}
\usepackage{tabularx}
\usepackage{afterpage}
\usepackage{lscape}
\usepackage{multirow}
\usepackage{multicol}
\usepackage{arydshln}
\usepackage{cancel}
\usepackage{makecell}
\usepackage{booktabs}

\usepackage{amsfonts,amscd}
\usepackage[T1]{fontenc}
\usepackage{lmodern}
\usepackage{nicefrac}
\usepackage{graphics,caption}
\usepackage[labelformat=simple]{subcaption}

\usepackage{arydshln}
\usepackage{lipsum}
\usepackage{amsfonts}
\usepackage{graphicx}
\usepackage{epstopdf}
\usepackage{mathtools}
\usepackage[numbers,sort&compress]{natbib}
\usepackage{stmaryrd,xparse}
\usepackage{doi}
\newcommand{\der}{{\rm d}}
\newcommand{\derd}{\delta}
\newcommand{\dtnp}{\Delta t_{n+1}}
\newcommand{\M}{\mathsf{M}}
\newcommand{\F}{\mathsf{F}}
\newcommand{\B}{\mathsf{B}}
\newcommand{\K}{\mathsf{K}}
\newcommand{\U}{\unk}
\newcommand{\Ua}{\tilde{u}_{h}}
\newcommand{\R}{\mathsf{R}}
\newcommand{\J}{\mathsf{J}}
\newcommand{\G}{\mathsf{G}}
\newcommand{\A}{\mathsf{A}}
\newcommand{\T}{\mathsf{T}}

\newcommand{\half}{\frac{1}{2}}

\newcommand{\testspace}{L^2(\Omega)}

\newcommand{\inflowboundary}{\Gamma_{\rm in}}
\newcommand{\domain}{\Omega}
\newcommand{\conv}{\boldsymbol{f}}
\newcommand{\rr}{\boldsymbol{r}}
\newcommand{\vv}{\boldsymbol{v}}
\newcommand{\normal}{\boldsymbol{n}}
\newcommand{\x}{\boldsymbol{x}}
\newcommand{\y}{\boldsymbol{y}}
\newcommand{\Conv}[2]{\conv'(#1) \cdot \gradient #2}
\newcommand{\uboundary}{u_D}
\newcommand{\mesh}{\mathcal{T}_h}
\newcommand{\element}[1][]{K_{#1}}
\newcommand{\fespace}{V_h}
\newcommand{\nodes}{\mathcal{N}_h}

\newcommand{\neighborhood}[1][]{\mathcal{N}_h(\domain_{#1})}

\newcommand{\shapef}[1][]{\varphi_{#1}}
\newcommand{\unk}{u_h}
\newcommand{\unknp}{\unk^{n+1}}
\newcommand{\unkn}{\unk^n}
\newcommand{\test}{v_h}
\newcommand{\feinterp}{\pi_h}
\newcommand{\smax}{\max{}_{\sigma}}
\newcommand{\detector}[1][]{\alpha_{#1}}
\newcommand{\graphl}{\ell}

\newcommand{\jump}[1]{\left\llbracket #1 \right\rrbracket}
\newcommand{\mean}[1]{%
	\sbox0{%
		\mathsurround=0pt % just for safety
		$\left\{\vphantom{#1}\right.\kern-\nulldelimiterspace$%
	}%
	\sbox2{\{}%
	\ifdim\ht0=\ht2
	\{\kern-.625\wd2 \{#1\}\kern-.625\wd2 \}%
	\else
	\left\{\kern-.7\wd0\left\{#1\right\}\kern-.7\wd0\right\}%
	\fi
}
\newcommand{\absn}[2][\varepsilon]{\left\vert #2 \right\vert_{1,#1}}
\newcommand{\absd}[2][\varepsilon]{\left\vert #2 \right\vert_{2,#1}}
\newcommand{\smthdetector}[1][]{ %
	\ifx\empty#1\empty%
	\alpha_{\varepsilon} %
	\else %
	\alpha_{\varepsilon,#1} %
	\fi}

\newcommand{\sdetectorAprox}[1][]{ %
	\ifx\empty#1\empty%
	\tilde{\alpha}_{\varepsilon} %
	\else %
	\tilde{\alpha}_{\varepsilon,#1} %
	\fi}

\newcommand{\gradient}{\boldsymbol\nabla}

\def\ij{{ij}}
\def\ji{{ji}}

\def\l2{{L^2}}
\def\h1{{H^1}}
\newcommand{\shock}[1][]{\alpha_{#1}}
\def\sym{{\rm sym}}
\newcommand{\symneigh}[1][]{\mathcal{N}^{\sym}_h(\domain_{#1})}

\newtheorem{theorem}{Theorem}[section]
\newtheorem{lemma}[theorem]{Lemma}

\newtheorem{corollary}[theorem]{Corollary}
\newtheorem{definition}[theorem]{Definition}

\newtheorem{remark}[theorem]{Remark}

\begin{document}

\title[Monotonicity-preserving differentiable nonlinear stabilization]{Monotonicity-preserving finite element schemes based on differentiable nonlinear stabilization}

\author[S. Badia]{ Santiago Badia$^\dag$}

\thanks{$\dag$ Universitat Polit\`ecnica de Catalunya, Jordi Girona1-3, Edifici C1, E-08034 Barcelona $\&$ Centre Internacional de M\`etodes Num\`erics en Enginyeria, Parc Mediterrani de la Tecnologia, Esteve Terrades 5, E-08860 Castelldefels, Spain E-mail: {\tt sbadia@cimne.upc.edu}.  SB was partially supported by the European Research Council under the FP7 Program Ideas through the Starting Grant No. 258443 - COMFUS: Computational Methods for Fusion Technology and the FP7 NUMEXAS project under grant agreement 611636. SB gratefully acknowledges the support received from the Catalan Government through the ICREA Acad\`emia Research Program. }

\author[J. Bonilla]{Jes\'us Bonilla$^\ddag$}
\thanks{$\ddag$ Universitat Polit\`ecnica de Catalunya, Jordi Girona1-3, Edifici C1, E-08034 Barcelona $\&$ Centre Internacional de M\`etodes Num\`erics en Enginyeria, Parc Mediterrani de la Tecnologia, Esteve Terrades 5, E-08860 Castelldefels, Spain E-mail: {\tt jbonilla@cimne.upc.edu}. JB gratefully acknowledges the support received from ''la Caixa'' Foundation through its PhD scholarship program.}

\date{\today}

%% \title{ Monotonicity-preserving finite element schemes based on differentiable nonlinear stabilization
%% \thanks{%
%% ACKNOWLEDGMENTS.}}

%% \author{ Santiago Badia\footnotemark[2]  \ \footnotemark[3]  \and Jes\'us Bonilla  \footnotemark[2]  \ \footnotemark[3] }

%% \renewcommand{\thefootnote}{\fnsymbol{footnote}}

%% \footnotetext[2]
%% {Centre Internacional de M\`etodes Num\`erics en Enginyeria (CIMNE)\@, Parc
%% Mediterrani de la Tecnologia, UPC, Esteve Terradas 5, 08860 Castelldefels,
%% Spain (\{sbadia,amartin,principe\}@cimne.upc.edu).} 
%% \footnotetext[3]{
%% Universitat Polit\`ecnica de Catalunya, Jordi Girona 1-3, Edifici C1, 08034
%% Barcelona, Spain.}
%% %\footnotetext[4]{
%% %Personal acknowledgments
%% %}

%% \renewcommand{\thefootnote}{\arabic{footnote}}

\begin{abstract} 

In this work, we propose a nonlinear stabilization technique for scalar conservation laws with implicit time stepping. The method relies on an artificial diffusion method, based on a graph-Laplacian operator. It is nonlinear, since it depends on a shock detector. Further, the resulting method is linearity preserving. The same shock detector is used to gradually lump the mass matrix. The resulting method is LED, positivity preserving, and also satisfies a global DMP. Lipschitz continuity has also been proved. However, the resulting scheme is highly nonlinear, leading to very poor nonlinear convergence rates. We propose a smooth version of the scheme, which leads to twice differentiable nonlinear stabilization schemes. It allows one to straightforwardly use Newton's method and obtain quadratic convergence. In the numerical experiments, steady and transient linear transport, and transient Burgers' equation have been considered in 2D. Using the Newton method with a smooth version of the scheme we can reduce 10 to 20 times the number of iterations of Anderson acceleration with the original non-smooth scheme. In any case, these properties are only true for the converged solution, but not for iterates. In this sense, we have also proposed the concept of projected nonlinear solvers, where a projection step is performed at the end of every nonlinear iterations onto a FE space of admissible solutions. The space of admissible solutions is the one that satisfies the desired monotonic properties (maximum principle or positivity). 

\end{abstract}

\maketitle

%\noindent{\bf 2010 Mathematics Subject Classification:} 35Q30; 65N30; 76N10.


\tableofcontents

%% \begin{keywords}
%% KEYWODS
%% \end{keywords}

%% \begin{AMS}
%%  AMS codes %65N55, 65F08, 65N30, 65Y05, 65Y20
%% \end{AMS}

%\pagestyle{myheadings}
%\thispagestyle{plain}

\section{Introduction}
\label{sec:intro}

Many partial differential equations (PDEs) satisfy some sort of maximum principle or positivity property. However, numerical discretizations usually violate these structural properties at the discrete level, with implications in terms of accuracy and stability, e.g., leading to non-physical local oscillations. 

It is well-understood now how to build methods that satisfy some sort of discrete maximum principle (DMP) based on explicit time integration combined with finite volume or discontinuous Galerkin schemes \cite{leveque_book_2002,cockburn_rungekutta_2001}. However, implicit time integration is preferred in problems with multiple scales in time when the fastest scales are not relevant. E.g., under-resolved simulations of multi-scale problems in time are essential in plasma physics \cite{kritz_fusion_2008}. Unfortunately, implicit DMP-preserving hyperbolic solvers are scarce and not so well developed.

In the frame of finite element (FE) discretizations, the local instabilities present in the solution of hyperbolic problems have motivated the use of so-called shock capturing schemes based on artificial diffusion (see, e.g., \cite{hughes_sc_1986}). These methods introduce nonlinear stabilization, in contrast with classical SUPG-type linear stabilization techniques \cite{hughes_supg_1979,hughes_gls_1989}. Since linear schemes are at most first-order accurate and highly dissipative \cite{godunov59}, recent research on FE techniques for conservation laws has focused on the development of less dissipative nonlinear schemes. Many of these ideas come from the numerical approximation of convection dominated convection-diffusion-reaction (CDR), where one encounters similar issues. The cornerstone of these methods is the design of a nonlinear artificial diffusion that vanishes in smooth regions and works on discontinuities or sharp layers. Many residual-based diffusion methods have been considered so far (see, e.g., \cite{donea_book_2003} and references therein). Most of these approaches have failed to reach DMP-preserving methods. A salient exception is the method by Burman and Ern \cite{burman_nonlinear_2002}, which satisfies a DMP under mesh restrictions.  
Recently, due to some interesting novel approaches in the field, the state-of-the-art in nonlinear stabilization has certainly advanced \cite{badia_monotonicity-preserving_2014,burman_monotonicity_2015,dmitri_kuzmin_new_2015,barrenechea_2016,kuzmin_gradient-based_2016,barrenechea_analysis_2016,burman_stabilized_2005}. 

Implicit FE schemes for hyperbolic problems rely on four key ingredients:

\begin{enumerate}
	
	\item The first ingredient is the definition of the \emph{shock detector} that only activates the nonlinear diffusion around shocks/discontinuities. Recent nonlinear stabilization techniques have been developed based on shock detectors driven by gradient jumps  \cite{burman_nonlinear_2007,badia_monotonicity-preserving_2014} or edge differences \cite{dmitri_kuzmin_new_2015,barrenechea_2016,kuzmin_gradient-based_2016}. The use of such schemes was proposed in \cite{burman_nonlinear_2007} for 1D problems and extended to multiple dimensions in \cite{badia_monotonicity-preserving_2014}. A salient property of the scheme in \cite{badia_monotonicity-preserving_2014} is that it is DMP-preserving, but it relies on the DMP of the Poisson operator, which is only true under stringent constraints on the mesh. Another salient feature of the gradient-jump diffusion approach in \cite{badia_monotonicity-preserving_2014} is the fact that it leads to so-called linearity preserving methods, i.e., the artificial diffusion vanishes for first order polynomials. This property is related to high-order convergence on smooth regions \cite{kuzmin_constrained_2009}. A modification of the nonlinear diffusion in \cite{dmitri_kuzmin_new_2015} that also satisfies this property is proposed in \cite{kuzmin_gradient-based_2016}.
	
	\item The second ingredient is the \emph{amount of diffusion} to be introduced on shocks, which is the amount of diffusion introduced in a first order linear scheme. In this sense, one can consider flux-corrected transport techniques \cite{kuzmin_flux_2002}.
	
	\item The third ingredient is the form of the \emph{discrete viscous operator}. In order to keep the DMP on arbitrary meshes, Guermond and Nazarov have proposed to use graph-theoretic, instead of PDE-based, operators for the artificial diffusion terms. This approach has been used in \cite{xu_monotone_1999,barrenechea_2016} (for the steady-state convection-diffusion-reaction problem)  and in \cite{guermond_second-order_2014} (for linear conservation laws) combined with artificial diffusion definitions similar to the one in \cite{guermond_maximum-principle_2014}. 
	
	\item The fourth ingredient is the \emph{perturbation of the mass matrix}, in order to satisfy a DMP. Full mass lumping is one choice, but it introduces an unacceptable phase error. For continuous FE methods, improved techniques can be found in \cite{guermond_correction_2013}. Alternatively, limiting-type strategies are used, e.g., in \cite{dmitri_kuzmin_new_2015,kuzmin_gradient-based_2016}.
	
	\item The method in \cite{barrenechea_2016} is Lipschitz continuous, which is needed for the well-posedness of the resulting nonlinear scheme.  However, in practice, all the methods presented above are still highly nonlinear, and nonlinear convergence becomes very hard and expensive. It leads to a fifth additional ingredient that has not been considered so far in much detail. In order to reduce the computational cost of these schemes, we consider the \emph{smoothing} of the nonlinear artificial diffusion, to make it differentiable up to some fixed order. The possibility to define smooth nonlinear schemes can improve the nonlinear convergence of the methods and make them practical for realistic applications. Further, the smoothing step enables advanced linearization strategies based on Newton's method. It also involves the development of efficient nonlinear solvers, e.g., based on the combination of Newton, line search, and/or Anderson acceleration techniques.
\end{enumerate}
All the results commented above are restricted to linear (or bilinear) FEs. We are not aware of the existence of high-order implicit DMP-preserving FE schemes. For explicit time integration and limiters, second order methods can be found in  \cite{guermond_second-order_2014}. The use of hp-adaptive schemes that keep first order schemes around shocks has been proposed in \cite{hierro_2016}.

In this work, we propose a novel nonlinear stabilization method that satisfies a DMP, positivity, and local extremum diminishing (LED) properties at the discrete level. It combines: (1) a novel shock detector related to the one in \cite{badia_monotonicity-preserving_2014}, which is simple and linearity preserving; (2) the graph-Laplacian artificial viscous term proposed in \cite{guermond_maximum-principle_2014}; (3) an edge FCT-type definition of the amount of diffusion (see \cite{dmitri_kuzmin_new_2015}); (4) a novel gradual mass lumping technique that exploits the same shock detector used for the artificial diffusion. We prove that the resulting method ticks all the boxes, i.e., it is total variation diminishing (TVD), DMP, positivity-preserving, linearity preserving, Lipschitz continuous, and introduces low dissipation. With regard to the last point, we prove that the amount of diffusion is the minimum needed in our analysis to prove the DMP. Further, we consider a novel approach to design a smoothed version of the resulting scheme that is twice differentiable. We prove that linear preservation is weakly enforced in this case, but all the other properties remain unchanged. Finally, we analyze the effect of the smoothing in the computational cost, and observe a clear reduction in the CPU cost of the nonlinear solver when using the smooth version of the method proposed herein while keeping almost unchanged the sharp layers of the non-smooth version. Future work will be focused on the entropy stability analysis of these schemes for nonlinear scalar conservation laws. A partial result in this direction is the proof of entropy stability for a related method when applied to the 1D Burger's equations (see \cite{burman_nonlinear_2007}).

This work is structured as follows. In Sect. \ref{sec:preliminaries} the continuous problem and its discretization using the FE method are presented. Sect. \ref{sec:method} contains the formulation of a novel nonlinear stabilization method. Sect. \ref{sec:properties} is devoted to the monotonicity analysis of the  proposed method. An alternative approach is presented in Sect. \ref{sec:l2stab}. Lipschitz continuity of the methods is proved in Sect. \ref{sec:lipschitz}. A differentiable version the previous method is presented in Sect. \ref{sec:dif_stab}. Sect. \ref{sec:solver} is devoted to nonlinear solvers. Different numerical experiments are introduced in Sect. \ref{sec:num_exp}. Finally, in Sect. \ref{sec:conclusions} we draw some conclusions.

\section{Preliminaries}\label{sec:preliminaries}

\subsection{The continuous problem}

Let $\domain\subset\mathbb{R}^d$ be a bounded domain, where $d$ is the space dimension, and $(0,T]$ the time interval. The scalar conservation equation reads: find $u(\x,t)$ such that
\begin{equation}
	\partial_t u + \gradient\cdot \conv (u) = g, \quad \hbox{on } \Omega \times (0,T],
\end{equation}
where $\conv\in\text{Lip}(\mathbb{R};\mathbb{R}^d)$ is the flux. It is also subject to the initial condition $u(\x,0)=u_0\in L^\infty(\Omega)$ and boundary condition $u(\x,t) = \uboundary(\x,t)$ on the inflow $\inflowboundary\doteq \{ (\x,t) \in \partial\domain \times (0,T] \mid \conv(\x,t ) \cdot \normal(\x,t) < 0\}$.  There exist a unique entropy solution $u$ of the above problem that satisfies the  entropy inequalities $\partial_t E(u) + \gradient\cdot\boldsymbol{F}(u)\leq 0$ for all convex entropies $E\in\text{Lip}(\mathbb{R};\mathbb{R})$ with its associated entropy fluxes $\boldsymbol{F}_i(u)=\int_0^u E^\prime(v) \conv^\prime_i (v)\, \text{d}v,\,1\leq i\leq d$ (see Kru\v{z}kov \cite{kruzkov_first_1970}).
Let us consider the weak form of this problem consists in seeking $u$ such that $u =\uboundary$ on $\inflowboundary \times (0,T]$ and 
\begin{equation}\label{eq:cont_prob}
	(\partial_t u, v) + (\Conv{u}{u},v) = (g,v) \quad \forall v\in \testspace,
\end{equation} 
almost everywhere in $(0,T]$, with $g\in\testspace$.

\subsection{Finite element spaces and meshes}\label{sec:femspace}

Let $\mesh$ be a conforming partition of $\domain$ into elements, $\element$. Elements can be triangles or quadrilaterals for $d=2$, or tetrahedrals or hexahedra for $d=3$. The set of interpolation nodes of $\mesh$ is represented by $\nodes$, whereas $\nodes(\element)$ denotes the set of nodes belonging to element $\element \in \mesh$. Moreover, $\domain_i$ is the macroelement composed by the union of the elements $\element\in\mesh$ such that $i\in\nodes(\element)$. $\neighborhood[i]$ denotes the set of nodes in that macroelement. The continuous linear FE space is defined as
\begin{equation}
	\fespace \doteq \left\{ v_h \in \mathcal{C}^0(\domain): \; v_h\vert_{\element} \in P_k(\element) \;\; \forall \element \in \mesh \right\}
\end{equation}
for triangular or tetrahedral elements (replacing $P_1(\element)$ by $Q_1(\element)$ for quadrilateral or hexahedral elements). $P_1(\element)$ (resp., $Q_1(\element)$) is the space of polynomials with total (resp., partial) degree less or equal to $1$. The nodal basis of $\fespace$ is written $\{\shapef[i]\}_{i\in\nodes}$, and the FE functions can be expressed as $v_h = \sum_{i\in\nodes} \shapef[i]v_i$, where $v_i$ is the value of $v_h$ at node $i$.

\subsection{The semi-discrete problem}

The semi-discrete Galerkin FE approximation of \eqref{eq:cont_prob} reads: find $\unk\in\fespace$ such that $\unk(\inflowboundary, t) = \feinterp(\uboundary)$ and
\begin{equation}\label{eq:disc_prob}
	(\partial_t \unk, \test) + (\Conv{\unk}{\unk},\test) = (g,\test) \quad \forall \test\in \fespace,
\end{equation}
for $t\in (0,T]$, with initial conditions $\unk(\cdot,0) = \feinterp(u_0)$. $\feinterp$ denotes a FE interpolation, e.g., the Scott-Zhang projector \cite{scott_finite_1990}. 

Using the notation $\M \unk \doteq (\unk, \cdot)$ and $\F(w_h)\unk \doteq  (\Conv{w_h}{\unk},\cdot)$ we can write problem \eqref{eq:disc_prob} in compact form as 
\begin{equation}\label{eq:semid}
	\M \partial_t \unk + \F(\unk) \unk = g  
\end{equation}
in $\fespace'$, i.e., the dual space of $\fespace$. Further, we define $\M_{\ij} \doteq  (\shapef[j], \shapef[i])$, $\F_{\ij}(\unk) \doteq (\Conv{\unk}{\shapef[j]},\shapef[i])$, and $g_i \doteq ( g , \shapef[i] )$.

In order to carry out the time discretization of \eqref{eq:semid}, let us consider a partition of the time domain $(0,T]$ into sub-intervals $(t^n,t^{n+1}]$, with $0 \doteq t^0 < t^1  < \ldots < t^N \doteq T$. We consider the Backward-Euler (BE) implicit time integrator to keep at the time-discrete level the monotonicity properties of the semi-discrete problem, leading to the discrete problem: given $\unk^0 \doteq \pi_h(u_0) \in \fespace$, compute for $n=1, \ldots, N-1$
\begin{equation}\label{eq:algdis}
	\M \derd_t \unknp + \F(\unknp)\unknp = g \quad \hbox{ in } \fespace',
\end{equation} 
where $\derd_t \unknp \doteq \dtnp^{-1} (\unknp - \unkn)$, and $\dtnp \doteq | t^{n+1} - t^n|$. Implicit strong stability preserving Runge-Kutta methods \cite{ketcheson_optimal_2009} also preserve the monotonic properties at the discrete level \cite{ketcheson_optimal_2009}, under some restrictions on the time step size. For the sake of brevity we consider the BE scheme.

Systems \eqref{eq:semid} and \eqref{eq:algdis} will be supplemented with additional stabilization terms to minimize the oscillations generated by the Galerkin FE approximation. 
Of particular interest are methods which provide solutions that satisfy the following property for all nodes, for zero forcing terms.
\begin{definition}[Local DMP]\label{def:lDMP}
	A solution $u \in \fespace$ satisfies the local DMP if 
	\begin{equation}
		u_i^{\min} \leq u_i \leq u_i^{\max}, \quad \ \hbox{where } \ u_i^{\max}\doteq \max_{j\in\neighborhood[i]\backslash\{i\}} u_j, \quad
		u_i^{\min}\doteq \min_{j\in\neighborhood[i]\backslash\{i\}} u_j.
	\end{equation}
\end{definition}
Actually, for steady problems, if this is satisfied for all $i \in \nodes$, then  the extrema will be at the boundary and there exist no local extrema. 

Furthermore, it is useful to define \emph{local extremum diminishing} (LED) methods for transient problems.
\begin{definition}[LED]\label{def:led}
	A method is called LED if for $g=0$ and any time in $(0,T]$, the solution satisfies 
	\begin{equation}
		{\der_t u_i} \leq 0 \; \text{if } u_i \text{ is a maximum and} \quad
		{\der_t u_i} \geq 0 \; \text{if } u_i \text{ is a minimum}.
	\end{equation}
	For time-discrete methods, the same definition applies, replacing $\der_t$ by $\derd_t$.
\end{definition}

\section{Nonlinear stabilization}\label{sec:method}

We want to design a linearity preserving LED method for stabilizing the scalar semi-discrete hyperbolic problem \eqref{eq:semid} (or the discrete problem \eqref{eq:algdis}), described in the previous section. As written above, this method is based on a graph-theoretic approach. Let us consider a nonlinear stabilization operator $\B(\unk): \fespace \rightarrow\fespace'$ and denote $\B_\ij(\unk) \doteq \langle \B(\unk) \shapef[j] , \shapef[i]  \rangle$. Particularly, we require that the stabilization term will satisfy the following properties (see also \cite{guermond_maximum-principle_2014}):
\begin{enumerate}
	\item compact support: $\B_\ij(\unk)=0$ if $j\notin \neighborhood[i]$ for any $\unk \in \fespace$,
	\item symmetry: $\B_\ij(\unk) = \B_\ji(\unk)$ for any $\unk \in \fespace$,
	\item conservation: $\sum_{j\neq i} \B_\ij(\unk) = - \B_{ii}(\unk)$ for any $\unk \in \fespace$,
	\item linear preservation: $\B(\unk)=0$ for any $\unk \in P_1(\Omega)$.
\end{enumerate}
To achieve this properties we define the nonlinear stabilization term
\begin{equation}\label{eq:stab}
	\langle \B(w_h) \unk,\test \rangle \doteq \sum_{i\in\nodes}\sum_{j\in\neighborhood[i]}\nu_{ij}(w_h) v_i u_j \graphl(i,j), \qquad \unk, \ \test \in \fespace,
\end{equation}
where the graph-theoretic Laplacian is defined as $\ell(i,j) \doteq 2\delta_{ij} - 1$, and the artificial diffusion computed as
\begin{equation}\label{eq:nu}
	\arraycolsep=1.4pt\def\arraystretch{1.4}
	\begin{array}{ll}
		\nu_{ij}(w_h) \doteq \max \left\{ \detector[i](w_h) \F_{ij}(w_h), \detector[j](w_h) \F_{ji}(w_h), 0 \right\} \qquad \text{for } i\neq j, \\
		\nu_{ii}(w_h) \doteq \displaystyle\sum_{\substack{j\in\neighborhood[i]\\ j\neq i}} \nu_{ij}(w_h),
	\end{array}
\end{equation}
where $\detector[i](\cdot)$ is the shock detector. We note that this choice leads to a symmetric stabilization operator $\B(w_h)$. 
In order to define the shock detector, let us introduce some notation. Let
$i\in\nodes$ be a node of the mesh, $\vv$ a vector field, and $w$ a scalar field. Let $\rr_{ij} = \x_j - \x_i$ be the vector pointing from nodes $i$ to $j$ in ${\nodes}$ and $\hat{\rr}_{ij} \doteq \frac{\rr_{ij}}{|\rr_{ij}|}$. Let $\x_\ij^\sym$ be the point at the intersection between the line that passes through $\x_i$ and $\x_j$ and $\partial\domain_i$ that is not $\x_j$ (see Fig. \ref{fig:usym}). The set of all symmetric nodes with respect to node $i$ is represented with $\symneigh[i]$.
We define $\rr_\ij^\sym \doteq \x_\ij^\sym - \x_i$, and $u_j^\sym \doteq \unk(\x_\ij^\sym)$. Then, one
can define the jump and the mean of the unknown gradient at node $i$ in direction $\rr_\ij$ as
\begin{align}\label{eq:jump}
	\jump{\gradient \unk}_\ij &\doteq \frac{u_j - u_i}{|\rr_\ij|} + \frac{u_j^\sym - u_i}{|\rr_\ij^\sym|}, \\ \label{eq:mean}
	\mean{|\gradient \unk\cdot \hat{\rr}_\ij|}_{ij} &\doteq \frac{1}{2}
	\left(\frac{|u_j-u_i|}{|\rr_\ij|}+\frac{|u_j^\sym-u_i|}{|\rr_\ij^\sym|}\right).
\end{align}
We note that the symmetric nodes and their corresponding values $u_j^\sym$ are used in the proof of the following results, Lemma \ref{lem:eqiv_barrenechea}, and Theorem \ref{th:lipschitz}, but \emph{not required in the implementation} of \eqref{eq:algebraic}. For triangular or tetrahedral meshes, since $\gradient\unk$ is constant, $u_j^\sym$ can be computed easily as
$$ 
u_j^{\sym} = \unk (\x_i) + \gradient\unk(\x_i) \cdot \rr_\ij^\sym.
$$
For quadrilateral or hexahedral structured (possibly adapted and nonconforming) meshes, $u_j^\sym$ is also easy to obtain since $j^{\rm sym}$ is already in $\neighborhood[i]$. It also applies for symmetric meshes, when a mesh is said to be symmetric with respect to its internal nodes if for any $i \in \nodes$ all symmetric nodes $j^\sym\in\symneigh[i]$ already belong to $\neighborhood[i]$.

\begin{figure}[h]
	\centering
	\includegraphics[width=0.25\textwidth]{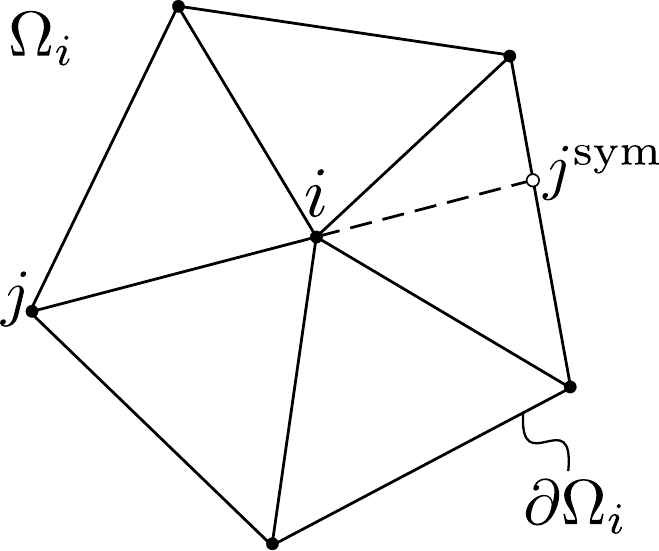}
	\caption{Representation of the symmetric node $j^{\rm sym}$ of $j$ with respect to $i$.}
	\label{fig:usym}
\end{figure}

Making use of these definitions, the proposed shock detector at node $i \in \nodes$ for a FE solution $\unk$ reads:
\begin{equation}\label{eq:alpha_lp}
	\detector[i](u_h) \doteq \left\{\begin{array}{cc}  \left[\frac{\left|{\sum_{j\in\neighborhood[i]} \jump{\gradient \unk}_{ij}}\right|}{\sum_{j\in\neighborhood[i]} 2\mean{\left|\gradient \unk \cdot \hat{\rr}_{ij}\right|}_{ij}} \right]^q & \text{if } \sum_{j\in\neighborhood[i]} \mean{\left|\gradient \unk\cdot \hat{\rr}_{ij}\right|}_{ij} \neq 0, \\
		0 & \text{otherwise},
	\end{array}\right. 
\end{equation}
for some $q \in \mathbb{R}^+$. We note that this shock detector is motivated from \cite{badia_monotonicity-preserving_2014}, where the directional nodal-wise jumps and mean values are first used for such purposes. For triangular or tetrahedral meshes, the only difference strives in the fact that the supremum over all $j\in\neighborhood[i]$ in both the numerator and denominator was used in \cite{badia_monotonicity-preserving_2014} instead of the sum. In the next lemma we show that in fact \eqref{eq:alpha_lp} detects extrema.

\begin{lemma}\label{lm:shock}
	The shock detector $\detector[i](\unk)$ defined in \eqref{eq:alpha_lp} is equal to 1 if $\unk$ has an extremum at point $\x_i$. Otherwise, $\detector[i](\unk)<1$ in general, and $\detector[i](\unk) = 0$  for $q = \infty$. 
\end{lemma}
\begin{proof}
	Using the fact that $\unk$ has an extremum at $\x_i$,
	\begin{align}
		\left | \sum_{j\in\neighborhood[i]} \jump{\gradient \unk}_{ij} \right| &= 
		\left | \sum_{j\in\neighborhood[i]}
		\frac{u_j - u_i}{|\rr_\ij|} + \frac{u_j^\sym - u_i}{|\rr_\ij^\sym|}\right| \\
		&= 
		\sum_{j\in\neighborhood[i]}
		\frac{\left|u_j - u_i\right|}{|\rr_\ij|} + \frac{\left|u_j^\sym - u_i\right|}{|\rr_\ij^\sym|} =
		\sum_{j\in\neighborhood[i]} 2\mean{ \left| \gradient \unk \cdot \hat{\rr}_{ij} \right|},
	\end{align}
	since $u_j-u_i$ has the same sign (or it is equal to zero) in all directions. It proves that $\detector[i](\unk) = 1$ on an extremum. In fact, if the solution does not have an extremum, these quantities neither can have the same sign nor be zero in all cases, and we only have
	\begin{equation}\label{eq:bounjump}
		\left | \sum_{j\in\neighborhood[i]} \jump{\gradient \unk}_{ij} \right| < 
		\sum_{j\in\neighborhood[i]}
		\frac{\left|u_j - u_i\right|}{|\rr_\ij|} + \frac{\left|u_j^\sym - u_i\right|}{|\rr_\ij^\sym|} =
		\sum_{j\in\neighborhood[i]} 2\mean{ \left| \gradient \unk \cdot \hat{\rr}_{ij} \right|}.
	\end{equation}
	Hence, $\detector[i](\unk) < 1$ when there is no extremum at $\x_i$. Moreover, for $q = \infty$, the shock detector vanishes in all the nodes that are not extrema.
\end{proof}

In addition to the nonlinear stabilization term $\B(\unk)$, it is necessary 
to do a mass matrix lumping to prove that the LED property is
satisfied. In the numerical analysis, it is enough to make this
approximation when testing against the shape functions corresponding to nodes related to extrema, which is identified by the shock
detector. Therefore, we propose the following stabilized semi-discrete version of \eqref{eq:disc_prob}:
\begin{equation}\label{eq:discrete}
	\arraycolsep=1.4pt\def\arraystretch{1.4}
	\begin{array}{rl}
		(1-\detector[i](u_h))(\partial_t \unk, \shapef[i]) & +\: \detector[i](u_h)(\partial_t u_i, \shapef[i]) + (\Conv{\unk}{\unk}, \shapef[i]) \\
		&+\displaystyle\sum_{j\in\neighborhood[i]} \nu_{ij}(\unk) v_i u_j l(i,j) = (g,\shapef[i]) \quad \text{for any } i \in \nodes,
	\end{array}
\end{equation}
with the definition of the shock detector \eqref{eq:alpha_lp} and the nonlinear artificial diffusion \eqref{eq:nu}. Thus, the definition of the mass matrix is nonlinear
\begin{equation}\label{eq:asymmetricmass}
	\M_\ij(\unk) \doteq (1-\detector[i](u_h))(\shapef[j], \shapef[i])  +\: \detector[i](u_h)( \delta_{ij}, \shapef[i]).
\end{equation}
It can be understood as a mass matrix with gradual lumping. Full lumping is only attained at extrema. 
Denoting $\K(\unk) \doteq \F(\unk) + \B(\unk)$, the stabilized problem \eqref{eq:discrete} can be expressed in compact form as 
\begin{equation}\label{eq:algebraic}
	\M(\unk) \der_t \unk + \K(\unk)\unk = g \quad \hbox{ in } \fespace'.
\end{equation} 
Analogously for the discrete problem \eqref{eq:algdis},
\begin{equation}\label{eq:algebraicdis}
	\M(\unknp) \derd_t \unknp + \K(\unknp)\unknp = g^{n+1} \quad \hbox{ in } \fespace'.
\end{equation}
Finally, let us note that the  shock detector \eqref{eq:alpha_lp} leads to the one
of Barrenechea and co-workers \cite{barrenechea_2016},
\begin{equation}\label{eq:alpha_lp_aprx}
	\tilde{\alpha}_i \doteq\left\{\begin{array}{cc} \left( \frac{\left\vert\sum_{j\in\neighborhood[i]}  u_i - u_j\right\vert}{\sum_{j\in\neighborhood[i]} \vert u_i - u_j\vert}\right)^q & 
		\text{if } \sum_{j\in\neighborhood[i]} \vert u_i - u_j\vert\neq 0, \\
		0 & \text{otherwise},
	\end{array}\right. 
\end{equation} 
when restricted to symmetric meshes of equilateral triangles.
\begin{lemma}\label{lem:eqiv_barrenechea}
	For a symmetric triangular mesh where all the edges have the same length, $\detector[i]$ in \eqref{eq:alpha_lp} is identical to $\tilde{\alpha}_i$ in \eqref{eq:alpha_lp_aprx}. 
\end{lemma}
\begin{proof} 
	For symmetric meshes, for every $j \in \neighborhood[i]$, $j^\sym \in \neighborhood[i]$.  So, we can group nodes in $\neighborhood[i]$ in pairs, getting
	$$
	2 \sum_{j \in \neighborhood[i]} (u_i - u_j) = \sum_{j \in \neighborhood[i]} (u_i - u_j + u_i - u_\ij^\sym).
	$$
	We proceed analogously for the mean value. Further, since $\rr_{ij}$ is identical for all $j \in \neighborhood[i]$ by assumption, we get
	$$
	\frac{\left|\sum_{j \in \neighborhood[i]}\jump{\gradient \unk}_\ij \right| }{2\sum_{j \in \neighborhood[i]} \mean{|\gradient \unk \cdot \hat{\rr}_\ij|}_\ij} 
	= \frac{\left| \sum_{j \in \neighborhood[i]} u_i - u_j \right|}{\sum_{j \in \neighborhood[i]} |u_i - u_j|}.
	$$
\end{proof}

For arbitrary symmetric meshes the methods only differ on the weights of the terms in the sums in \eqref{eq:alpha_lp} and all the required properties stated in \eqref{eq:conditions} are readily satisfied for the use of the shock detector in \eqref{eq:alpha_lp_aprx}. In general meshes, the shock detectors are different, and the one in \eqref{eq:alpha_lp_aprx} is not linearity preserving.

\section{Monotonicity properties}\label{sec:properties}
In the sequel, we prove that the scheme \eqref{eq:discrete} is LED. First, we define a set of necessary conditions on the nonlinear discrete operators that lead to LED schemes. They are the nonlinear extension of the ones  for linear systems (see, e.g., \cite{dmitri_kuzmin_new_2015}). 
\begin{theorem}\label{th:properties}
	The semi-discrete problem \eqref{eq:algebraic} is LED if $g(\x) = 0$ in $\domain$ and, for every node $i \in \nodes$  such that  $u_i$ is a local extremum, it holds:
	\begin{align}
		& \M_{ij}(\unk) \doteq \delta_\ij m_i, \hbox{ with } m_i > 0, \label{eq:conditions}\\
		& \K_\ij(\unk) \leq 0\; \ \forall i\neq j,  \ \hbox{and} \ \sum_{j \in \neighborhood[i]} \K_\ij(\unk) = 0.
	\end{align}
	Moreover, for $g(\x) \leq 0$ (resp. $g(\x) \geq 0$) in $\domain$ and for all $i \in \nodes$ such that $u_i$ is a local maximum (resp. minimum), if \eqref{eq:conditions} holds the maximum (resp. minimum) is diminishing (resp. increasing). These results are also true for the discrete problem \eqref{eq:algebraicdis}. Furthermore, the discrete problem \eqref{eq:algebraicdis} is positivity-preserving for $g = 0$ and $u_0 \geq 0$.
\end{theorem}
\begin{proof}
	Let us start proving the LED property. If $u_i$ is a maximum, from \eqref{eq:algebraic}, conditions in \eqref{eq:conditions}, and the fact that $\alpha_i(\unk) = 1$, we have:
	$$
	g_i = m_i \der_t {u}_i + \sum_{j \in \neighborhood[i]} \K_{\ij}(\unk) u_j \geq m_i \der_t u_i + \sum_{j \in \neighborhood[i]} \K_{\ij}(\unk) u_i  =  m_i \der_t{u}_i,
	$$
	for $m_i \doteq \int_\Omega \shapef[i] \der \Omega$. As a result, $\der_t {u}_i \leq 0$ and thus LED. We proceed analogously for the minimum. The proof is analogous for the discrete problem with BE time integration.
	
	Next, we prove positivity. Let us consider that at some time step $m$ the solution becomes negative, and consider the node $i$ in which the minimum value is attained. Using the previous result for a minimum at the discrete level, we have that $\derd_t u_i^m \geq 0$ and thus $u_i^m \geq u_i^{m-1}$. It leads to a contradiction, since $u_i^{m-1} \geq 0$. Thus, the solution must be positive at all times.
\end{proof}

\begin{theorem}[LED]\label{th:led}
	The semi-discrete  (resp., discrete) problem \eqref{eq:algebraic} (resp., \eqref{eq:algebraicdis}) leads to solutions $\unk \in \fespace$ that enjoy the LED property in Def. \ref{def:led} for any $q \in \mathbb{R}^+$. 
\end{theorem}
\begin{proof}
	Assume $\unk$ reaches an extremum on $i\in\nodes$. Then $\detector[i](\unk)=1$ and $\M_\ij(\unk) \der_t u_j = m_i \der_t u_i$ with $m_i = \int_\Omega \shapef[i]$. On the other hand, taking into account the definition of $\nu_{ij}(\unk)$ in \eqref{eq:nu}, the convective term for $j\neq i$ reads
	\begin{align}
		\K_{ij} (\unk) = \F_{ij}(\unk) - \max\left\{\F_{ij}(\unk), \detector[j](\unk) \F_{ji}(\unk) ,0\right\} \leq 0.
	\end{align}
	Using the fact that $\sum_{j\in\neighborhood[i]}\F_{ij}(\unk)= (\Conv{\unk}{1},\shapef[i]) = 0$, the definition of $\nu_{ii}(\unk)$, and \eqref{eq:stab}, we have
	\begin{align}
		\K_{ii} (\unk) &= \F_{ii}(\unk) + \sum_{j\in\neighborhood[i]\backslash\{i\}} \max\left\{\F_{ij}(\unk), \detector[j](\unk) \F_{ji}(\unk) ,0\right\} \\ & = \sum_{j\in\neighborhood[i]\backslash\{i\}} - \F_{ij}(\unk) + \sum_{j\in\neighborhood[i]\backslash\{i\}} \max\left\{\F_{ij}(\unk), \detector[j](\unk) \F_{ji}(\unk) ,0\right\} \\ & =  - \sum_{j\in\neighborhood[i]\backslash\{i\}} \K_{ij}(\unk).
	\end{align}
	Therefore it is clear that the conditions stated in Theorem \ref{th:properties} hold, thus the method is LED. The discrete case is proved analogously.
\end{proof}
\begin{corollary}[DMP]
	The discrete problem \eqref{eq:algebraicdis} 
	leads to solutions that satisfy the local DMP property in Def. \ref{def:lDMP} at every $t^n$, for $n = 1, \ldots, N$.
\end{corollary}
\begin{proof}
	If the maximum (resp., minimum) at time $t^n$ is on a node whose value is not on the Dirichlet boundary, it is known from the LED property in Theorem \ref{th:led} that it is bounded above (resp., below) by the maximum (resp., minimum) at the previous time step value. By induction, it will be bounded by the maximum (resp., minimum) at $t=0$. Alternatively, the maximum or minimum is on the Dirichlet boundary. It proves the result.
\end{proof} 

\begin{theorem}
	The diffusion defined in \eqref{eq:nu} is the one that introduces the minimum amount of numerical dissipation $\langle \B(\unk) \unk , \unk \rangle$ required to satisfy \eqref{eq:conditions} when $q = \infty$.
\end{theorem}

\begin{proof}
	Using the definition of the graph-Laplacian, the amount of dissipation introduced by the nonlinear stabilization is
	$$
	\langle \B(\unk) \unk , \unk \rangle = \sum_{i \in \nodes} \sum_{j \in \neighborhood } \nu_\ij(\unk) (u_i - u_j)^2.
	$$ Let us consider two connected nodes, i.e., $i, j \in \nodes$ and $j
	\in \neighborhood[i]$. If neither $i$ nor $j$ are extrema, then
	$\detector[i](\unk) = \detector[j](\unk) = 0$ and $\nu_\ij = 0$. Let
	us assume (without loss of generality) that $\unk$ has an extremum at
	$i$.  If $u_i = u_j$, the dissipation is independent of the expression
	for $\nu_{\ij}$.  If $u_i > u_j$, $\shock[j] = 0$ (since $q =
	\infty$). Thus, $\nu_\ij = -\max\{\F_\ij(\unk),0\}$. If $\F_\ij(\unk)
	\leq 0$, no dissipation is introduced. If $\F_\ij(\unk) > 0$, then the
	diffusion introduced by the method is $-\F_\ij(\unk)$ and
	$\K_\ij(\unk) = 0$.
	
	Let us assume that we have a method that is less dissipative than the one proposed herein. Based on the previous analysis, there exists a pair of connected nodes such that $u_i > u_j$ and the dissipation introduced is smaller than $- \F_{ij}(\unk)$, for $\F_\ij(\unk) > 0$. As a result, $\K_\ij(\unk) > 0$. Thus, the properties in \eqref{th:properties} do not hold. It proves the theorem.
\end{proof}

Furthermore, it can be proved that the above method \eqref{eq:algebraic} (also \eqref{eq:algebraicdis}) is linearly preserving. In addition, using \eqref{eq:alpha_lp_aprx} instead, the method is still linearly preserving for symmetric meshes.
\begin{theorem}[Linearity preservation]\label{th:lp}
	Let $\unk$ be a continuous first order FE approximation of $u\in P_1(\domain)$, then the semi-discrete and discrete problems \eqref{eq:algebraic} and \eqref{eq:algebraicdis}, respectively, are linearity preserving, in the sense that the Galerkin problem and the stabilized one are identical.
\end{theorem}
\begin{proof}
	If $\unk\in P_1(\domain)$, then it is obvious that $\gradient\unk$ is constant. Thus, $\jump{\gradient \unk}_{ij}=0 $ for any direction $\rr_{ij}$, and $\detector[i](\unk)=0$ for any $i \in \nodes$. Therefore, recalling \eqref{eq:nu}, it is easy to see that $\nu_{ij}=0$ for any $i,j \in \nodes$. Thus, the nonlinear stabilization and gradual lumping terms vanish and the Galerkin scheme is recovered.
\end{proof}

\section{Symmetric mass matrix stabilization}\label{sec:l2stab}

The nonlinear mass matrix that has been considered in  \eqref{eq:asymmetricmass} is nonsymmetric by construction. In any case, we can easily consider a symmetric version of the method.

Another alternative strategy to the nonlinear mass matrix definition in \eqref{eq:asymmetricmass} is to consider the fully discrete problem \eqref{eq:algebraicdis}, keeping the mass matrix at the current time step as a reaction term, leading to the following expression of the artificial diffusion
\begin{equation}\label{eq:nuConsMass}
	\arraycolsep=1.4pt\def\arraystretch{2}
	\begin{array}{ll}
		\tilde{\nu}_{ij}(w_h)\doteq\nu_{ij}(w_h) + \frac{1}{\Delta t}\max \left\{\detector[i]\M_\ij,0, \detector[j]\M_\ji\right\} \qquad\text{for } i\neq j, \\
		\tilde{\nu}_{ii}(w_h) \doteq \displaystyle\sum_{\substack{j\in\neighborhood[i]\\ j\neq i}} \tilde{\nu}_{ij}.
	\end{array}
\end{equation}

Let us consider another notion of DMP property. 
\begin{definition}[Global DMP]\label{def:gDMP}
	A solution satisfies the global DMP if given $(\x,t)$ in $\domain\times (0,T]$
	\begin{equation}
		\min_{(\boldsymbol{y},\overline{t})\in\Gamma} u(\boldsymbol{y},\overline{t})  \leq u(\x,t) \leq \max_{(\boldsymbol{y},\overline{t})\in\Gamma} u(\boldsymbol{y},\overline{t}) 
	\end{equation}
	where $\Gamma \doteq \Omega \times \{0\} \cup \inflowboundary$. \\
\end{definition}  
It is easy to check that the global DMP is a consequence of the local DMP and LED properties.

It is possible to prove that the modified method with BE time integration satisfies the global DMP in Def. \ref{def:gDMP}. Linear preservation can also be easily checked.
\begin{theorem}[Global DMP]\label{th:gdmp}
	Let $\unk$ be a continuous first order FE approximation of $u$. Then, the BE time discretization of problem \eqref{eq:disc_prob} with $g= 0$, stabilized with \eqref{eq:stab}, and using \eqref{eq:nuConsMass} as artificial diffusion, satisfies the global DMP property in Def. \eqref{def:gDMP} for any $q\in\mathbb{R}^+$.
\end{theorem}
\begin{proof}
	Let us denote by $\K(u)$ and $\tilde{\K}(u)$ the stabilized matrix with the artificial diffusion computed with \eqref{eq:nu} and \eqref{eq:nuConsMass}, respectively. Assume $\unk$ reaches a maximum on $\x_i\in\domain\backslash\inflowboundary$. Then $\detector[i]=1$, and we have:
	$$
	\M_\ij(\unk) u_j + \tilde{\K}_{\ij} (\unk) u_j = m_i u_i + {\K}_{\ij} (\unk) u_j,
	$$
	where we have used the fact that $\max \left\{\detector[i]\M_\ij,0, \detector[j]\M_\ji\right\} = \M_\ij$. Thus, the equation related to the test function  $\shapef[i]$ leads to 
	\begin{equation}
		\sum_{j\in\neighborhood[i]} \frac{\M_{ij}}{m_i} u_j^{n} = u_i^{n+1}+\sum_{j\in\neighborhood[i]}\frac{\K_{ij} (\unk)}{m_i}  u_j^{n+1} \geq 
		u_i^{n+1}+\sum_{j\in\neighborhood[i]}\frac{\K_{ij} (\unk)}{m_i} u_i^{n+1} = u_i^{n+1}.
	\end{equation}
	Note that $\frac{\M_{ij}}{m_i}>0$, and $\sum_{j\in\neighborhood[i]} \frac{\M_{ij}}{m_i} = 1$. Hence $u_i^{n+1}$ is smaller or equal to a convex combination of $u_j^n$, for $j\in\neighborhood[i]$, and thus it is bounded above by the largest of these values. As a result, $u_h^{n+1}(\x) \leq \max \{ \max_{\y \in \Omega} u_h^n(\y), \max_{(\y,t^{n+1}) \in \Gamma_{\rm in}} u_D(\y,t^{n+1})  \}$. Using a recursion argument, we prove the upper bound. We proceed analogously for the case lower bound. It proves the theorem.
\end{proof}

\section{Lipschitz continuity}\label{sec:lipschitz}

In the next, we want to prove the Lipschitz continuity of the nonlinear operator at every time step, i.e., $\T: \fespace \rightarrow \fespace'$ defined as 
$$
\T(\unk) \doteq  \dtnp^{-1} \M(\unk) \unk + \K(\unk)\unk - g - \dtnp^{-1} \M(\unk) \unkn.
$$
In order to prove the Lipschitz continuity of $\T(\cdot)$, we must deal with the nonlinear stabilization and gradual mass lumping terms. The Galerkin terms can be handled using the fact that $\conv\in\text{Lip}(\mathbb{R};\mathbb{R}^d)$.

Let us introduce the following semi-norm generated by the graph-Laplacian operator 
$$| w |_\ell \doteq \sqrt{ \half \sum_{{i \in \nodes }}
	\sum_{{j\in\neighborhood[i]}} (w_i - w_j)^2. }$$  
Further, we define
$|\beta|$ as the supremum of $ |\conv(v) |$ for $v \in \fespace^{\rm adm}$, where 
$\fespace^{\rm adm} \subset \fespace$ is the subspace of functions that satisfy the global DMP in Def. \ref{def:gDMP}. 

\begin{theorem}\label{th:lipschitz}
	Let us consider a non-degenerate partition $\mathcal{T}_h$. Given $\unkn \in \fespace$ and $g \in \fespace'$, the nonlinear operators $\B(\cdot): \fespace \rightarrow \fespace'$ and $\M(\cdot): \fespace \rightarrow \fespace'$ are Lipschitz continuous in $\fespace^{\rm adm}$ for $q \in \mathbb{N}^+$, since they satisfy
	$$
	\langle \B(u) - \B(v) , z \rangle  \leq q h^{d-1} | \beta | |u-v|_\ell |z|_\ell
	, \qquad \hbox{for any} \ z \in \fespace,
	$$
	$$
	\langle \M(u) - \M(v) , z \rangle  \leq C (q h^{\frac{d}{2}} |u-v|_\ell + \| u -v \|)\| z \|, \qquad \hbox{for any} \ z \in \fespace.
	$$
\end{theorem}
\begin{proof}
	The proof of the theorem is included in Appendix \ref{proof_continuity}.
\end{proof}

\section{Differentiable stabilization}\label{sec:dif_stab}
The previous nonlinear system is Lipschitz continuous, which improves
the convergence of the nonlinear iterations. In fact, assuming that we
supplement \eqref{eq:cont_prob} with a diffusive term, 
existence and uniqueness can be proved in the diffusive regime (see
\cite{barrenechea_2016}). However, even using Anderson acceleration nonlinear convergence can be very hard 
(see \cite{dmitri_kuzmin_new_2015,kuzmin_gradient-based_2016} and
Sect. \ref{sec:num_exp}). 

Based on these observations, we want to develop methods that lead to at least twice differentiable operators, i.e., $\frac{\partial^2 \T(\unk)}{\partial^2 \unk}\in
\mathcal{C}^0$, using the previous framework. This allows the usage of the Newton method to
linearize the system, and reduces the required number of nonlinear
iterations. Smoothness is achieved by substituting the
non-differentiable functions of the previous formulation with smooth
approximations. 

In order to end up with a twice differentiable method, we propose to use the following artificial diffusion:
\begin{equation}\label{eq:soft_nu}
	\arraycolsep=1.4pt\def\arraystretch{1.6}
	\begin{array}{rl}
		\nu_{ij}\doteq& \smax \left\{\smax \left\{\smthdetector[i](\F_\ij(w_h)), \smthdetector[j]\F_\ji(w_h)\right\},0\right\}, \quad\text{for } i\neq j,\\ \nu_{ii}\doteq& \displaystyle\sum_{\substack{j\in\neighborhood[i]\\ j\neq i}} \nu_{ij}.
	\end{array}
\end{equation}
The function $\smax(\cdot)$ is a regularized maximum function
\begin{equation}\label{eq:smax}
	\smax\{x,y\} \doteq \frac{ \absn[\sigma]{x -y}}{2} +\frac{x +y}{2},
\end{equation}
where $\absn[\sigma]{x}\doteq\sqrt{x^2+\sigma}$ is a smooth approximation of the absolute value. In order to keep dimensional consistency, $\sigma$ should be a small parameter of order $\mathcal{O}\left(|\beta|^2 \ell^{2(d-1)}\right)$, where $\ell$ is a characteristic length of the problem.
Let us define the smooth limiter function $f(x) \in \mathcal{C}^2$ that will be used in the definition of $\smthdetector$,
\begin{equation}\label{eq:alpha_limit}
	f(x)\doteq \left\{\begin{array}{clc}
		2x^4 - 5x^3 + 3x^2 + x & \text{if} & x<1 \\
		1 & \text{if} & x\geq 1 \end{array}\right. .
\end{equation}
This function is used to smoothly limit the value of $x$ up to 1. Further, let us define another smooth approximation of the absolute value, namely 
$$\absd[\varepsilon]{x}\doteq \frac{x^2}{\sqrt{x^2 + \varepsilon}}.$$ Finally, the shock detector is defined as
\begin{equation}\label{eq:alpha_smooth}
	\smthdetector[i](u_h) \doteq
	\left[f\left(\frac{\absn{\sum_{j\in\neighborhood[i]} \jump{\gradient \unk}_{ij}} + \gamma}{\sum_{j\in\neighborhood[i]}2\mean{\absd{\gradient \unk\cdot \hat{\rr}_{ij}}}_{ij} + \gamma}\right)\right]^q ,
\end{equation}
where $\gamma$ is a small parameter that prevents division by zero.

It has been proved in Lemma \ref{lm:shock} that $\detector[i]$ equals 1 when $i$ is an extremum in $\domain_i$. Let us prove that this is still true for $\smthdetector[i]$.
\begin{lemma}\label{lem:smth_alpha} If $\unk$ has an extremum on $i\in\nodes$ then $\smthdetector[i](\unk) = 1$.
\end{lemma}
\begin{proof}
	It is clear that $f(x)$ equals 1 for $x\geq 1$, then the proof reduces to check that
	\begin{equation}
		\absn{\sum_{j\in\neighborhood[i]} \jump{\gradient \unk}_{ij}} + \gamma \geq \sum_{j\in\neighborhood[i]}2\mean{\absd{\gradient \unk\cdot \hat{\rr}_{ij}}}_{ij} + \gamma .
	\end{equation}
	Taking into account that
	\begin{equation}
		\sqrt{x^2 + \varepsilon} = \absn{x} > |x| \geq \absd{x} = \frac{x^2}{\sqrt{x^2 + \varepsilon}},
	\end{equation}
	and the fact that $u_j-u_i$ has the same sign (or it is equal to zero) in all directions, it is easy to see that
	\begin{align}
		\absn{ \sum_{j\in\neighborhood[i]} \jump{\gradient \unk}_{ij}}  &= 
		\absn{\sum_{j\in\neighborhood[i]}
			\frac{u_j - u_i}{|\rr_\ij|} + \frac{u_j^\sym - u_i}{|\rr_\ij^\sym|}} \\
		&\geq 
		\left | \sum_{j\in\neighborhood[i]}
		\frac{u_j - u_i}{|\rr_\ij|} + \frac{u_j^\sym - u_i}{|\rr_\ij^\sym|}\right| 
		\\
		&=
		\sum_{j\in\neighborhood[i]}
		\frac{\left|u_j - u_i\right|}{|\rr_\ij|} + \frac{\left|u_j^\sym - u_i\right|}{|\rr_\ij^\sym|} \geq
		\sum_{j\in\neighborhood[i]} 2\mean{ \left| \gradient \unk \cdot \hat{\rr}_{ij} \right|} \\
		&\geq \sum_{j\in\neighborhood[i]} 2\mean{ \absd{ \gradient \unk \cdot \hat{\rr}_{ij}} }.
	\end{align}
	It proves that $\smthdetector[i](\unk) = 1$ on an extremum. In fact, if the solution does not have an extremum, these quantities neither can have the same sign nor be zero in all cases. Since 
	\begin{equation}
		\left | \sum_{j\in\neighborhood[i]} \jump{\gradient \unk}_{ij} \right|  = \lim_{\varepsilon\rightarrow 0} \absn{\sum_{j\in\neighborhood[i]} \jump{\gradient \unk}_{ij} }
	\end{equation}
	and
	\begin{equation}
		\sum_{j\in\neighborhood[i]} 2\mean{ \left| \gradient \unk \cdot \hat{\rr}_{ij} \right|} =
		\lim_{\varepsilon\rightarrow 0} \sum_{j\in\neighborhood[i]} 2\mean{ \absd{\gradient \unk \cdot \hat{\rr}_{ij}}},
	\end{equation}
	bound \eqref{eq:bounjump} leads to the fact that $\lim_{\varepsilon\rightarrow 0} \smthdetector[i](\unk) < 1$ when there is no extremum on $i$. 
\end{proof}

It is straightforward to check the following results. 
\begin{corollary}
	System \eqref{eq:algebraic} with the definition of the shock detector \eqref{eq:alpha_smooth} and artificial diffusion \eqref{eq:soft_nu} is LED and satisfies the local DMP. The method tends to a linearly preserving scheme as $\gamma \to 0$.
\end{corollary}
\begin{proof}
	From lemma \ref{lem:smth_alpha} and the definition of the regularized maximum \eqref{eq:smax} it is easy to see that artificial diffusion in \eqref{eq:soft_nu} is greater or equal to the one in \eqref{eq:nu}. Hence, Theorem \ref{th:led} still holds. The linearity preservation is straighforward.
\end{proof}

\begin{remark}
	Note that the smoothed shock detector is not linearly preserving because $\smthdetector[i]$ will never be zero. However, for regions where $\unk$ is constant the gradient is zero, thus the solution is not affected. In the case of $\unk\in P_1(\domain)$, but not constant, $\smthdetector[i]$ goes to zero with $\gamma$. Values of $\gamma$ of order $10^{-8}$ (or even smaller) have been considered in the numerical experiments section with good nonlinear convergence properties. Thus, the linearity preservation is virtually preserved in practice.
\end{remark}

As in the previous section, when restricted to symmetric meshes, the following approximation (similar to the one in Barrenechea et al. \cite{barrenechea_2016}) of \eqref{eq:alpha_smooth} maintains the same properties
\begin{equation}\label{eq:alpha}
	\sdetectorAprox[i] \doteq 
	\left[f\left(\frac{\absn[\varepsilon^*]{\sum_{j\in\neighborhood[i]} u_i-u_j}+\gamma^*} {\sum_{j\in\neighborhood[i]}\absd[\varepsilon^*]{u_i-u_j}+\gamma^*}\right)\right]^q ,
\end{equation}
with $\varepsilon^*\sim\mathcal{O} (h^2\varepsilon)$ and $\gamma^*\sim \mathcal{O}(h\gamma)$.

\section{Nonlinear Solvers}\label{sec:solver}
In this section the methods used for solving the system of nonlinear equations resulting from the above formulation \eqref{eq:algebraicdis} with the artificial diffusion defined in \eqref{eq:soft_nu} is discussed. Taking advantage of the differentiability of the stabilization described in Sect. \ref{sec:dif_stab}, Newton's method is used for the smooth version of the method. In addition, we use fixed point iterations with Anderson acceleration to compare against Newton's method performance. 
In order to define the schemes, it is useful to write the time-discrete problem \eqref{eq:algebraicdis} as
\begin{equation}
	\A(\U^{n+1}) \U^{n+1} = \G
\end{equation}
where $\G$ is the force vector. Let $\J(\U^{n+1})\doteq \frac{\partial \T(\U^{n+1})}{\partial \U^{n+1}}$ be the Jacobian. 

Since the above problem is nonlinear we will solve it iteratively. We denote by $\U^{k,n+1}$ the $k$-th iteration of $\U$ at time step $n+1$. Let us define some auxiliary variables used in the definition of the algorithms: $m$ denotes the number of previous nonlinear iterations used in Anderson acceleration, $s$ is the slope resulting form fitting the last $m$ nonlinear errors, $s_{\min}$ is the minimum slope allowed before increasing the relaxation, $\omega$ is the relaxation parameter, $\omega_{\min}$ is its allowed minimum, $k_{\max}$ is the maximum nonlinear iterations allowed, $tol$ is the nonlinear tolerance, and $nlerr$ is the nonlinear error. 

For the non-differentiable methods in Sect. \ref{sec:method} we use Picard linearization with Anderson acceleration (see Alg. \ref{alg:anderson}). Our particular implementation also includes a simple convergence rate test, where it is decided if the relaxation parameter should be reduced or not. This improves the global convergence rate and the robustness of the method. Moreover, we add a projection onto $\fespace^{\rm adm}$ to ensure that the global DMP in Def. \ref{def:gDMP} is satisfied at all nonlinear iterations. This step is of special interest in the case of solving for variables that cannot become negative, e.g., the density. In this case, the projection onto the space of admissible solutions is performed truncating the obtained solution. However, more sophisticated methodologies can be also applied but at a higher computational cost.

\begin{algorithm}[h]
	\caption{Fixed point iterations with relaxed Anderson acceleration}\label{alg:anderson}
	\KwIn{$\U^{0,n+1}$, $m$, $s_{\min}$, $\omega_{\min}$, $tol$, $\A$, $\G$, $k_{\max}$}
	\KwOut{$\U^{k,n+1}$,$k$}
	$k=1$, $nlerr^1=tol$\\
	\While{($nlerr^k\geq tol$) and ($k<k_{\max}$)}{
		Set $m^k=\min(k,m)$ \\
		Solve $\A(\U^{k,n+1})\Ua^{k,n+1} = \G$ \\
		Compute $r^{k,n+1} = \Ua^{k,n+1} - \U^{k,n+1}$ \\
		Minimize $\|\sum_{i=1}^{m^k} \xi_i^k r^{k-m^k+i,n+1}\|$ with respect to $\xi_i^k$ subject to $\sum_{i=1}^{m^k} \xi_i^k=1$\\
		Set $\U^{k+1,n+1}=(1-\omega_k)\sum_{i=1}^{m^k} \xi_i^k \U^{k-m+i,n+1} + \omega_k\sum_{i=1}^{m^k} \xi_i^k \Ua^{k-m^k+i,n+1}$ \\
		Project $\U^{k+1,n+1}$ to $\fespace^{\rm adm}$ \\
		Set $nlerr^k=\frac{\|\U^{k+1,n+1}-\U^{k,n+1}\|}{\|\U^{k+1,n+1}\|}$ \\
		Compute the slope ($s$) of $\{nlerr^i\}$ with $k\geq i \geq k-m^k$\\
		\eIf{$(s<s_{\min})$ and $(\omega>\omega_{\min})$}{
			Set $\omega_{k+1}=\omega_k-0.1$
		}{
		Set $\omega_{k+1}=\omega_k$
	}
	Update $k= k+1$
}
\end{algorithm}

For the differentiable method, Newton's linearization is used (see Alg. \ref{alg:newton}). In addition, we supplement it with the line search method to improve robustness. We use numerical 1D minimization of the residual norm up to a tolerance of $10^{-4}$ for the line search method.
Following the same approach in Alg. \ref{alg:anderson}, a projection to the FE space of admissible solutions can be performed in Alg. \ref{alg:newton}. As said before, this step ensures that for all nonlinear iterations the solution satisfies the global DMP. The numerical experiments in the next section show that the modified method keeps quadratic convergence, even though we do not have a theoretical analysis.

\begin{algorithm}[h]
	\caption{Newton's method + Line search}\label{alg:newton}
	\KwIn{$\U^{0,n+1}$,$\U^n$, $tol$, $\J$, $\R$, $k_{\max}$}
	\KwOut{$\U^{k,n+1}$,$k$}
	$k=1$, $nlerr^1=tol$ \\
	\While{($nlerr^k\geq tol$) and ($k<k_{\max}$)}{
		Solve $\J(\U^{k,n+1})\Delta \U^{k,n+1} = -\T(\U^{k,n+1})$ \\
		Minimize $\|\T(\U^{k,n+1}+\xi^k\Delta \U^{k,n+1})\|$ with respect to $\xi\in[0,1]$ \\
		Set $\U^{k+1,n+1}=\U^{k,n+1} + \xi^k\Delta \U^{k,n+1}$ \\
		Project $\U^{k+1,n+1}$ to $\fespace^{\rm adm}$ \\
		Set $nlerr^k=\frac{\|\xi^k\Delta \U^{k,n+1}\|}{\|\U^{k+1,n+1}\|}$\\
		Update $k= k+1$
	}
\end{algorithm}

\section{Numerical Experiments}\label{sec:num_exp}
\subsection{Steady problems}
First, in order to test the previous formulation, the convergence to a smooth solution is analyzed. For this purpose, the following equation is solved
\begin{equation}\label{eq:steady_transport}
	\begin{array}{rlcl}
		\gradient\cdot(\mathbf{v}u) &= 0 &\text{ in }& \domain=[0,1]\times[0,1], \\
		u &= u_D &\text{ on } & \inflowboundary,
	\end{array}
\end{equation}
with $\mathbf{v}(x,y) \doteq \left(1,\, 0\right)$, and inflow boundary conditions $u_D=y-y^2$ on $\partial\domain\backslash\{x=1\}$. This problem consists in the transport of the parabolic profile along the $x$ direction, which has the analytical solution $u(x,y) =y-y^2$.

Fig. \ref{fig:convergence} shows the convergence rates using the previously defined formulation (\eqref{eq:algebraicdis} with \eqref{eq:soft_nu}), and the Galerkin formulation. To perform this test, an initial mesh of $12\times 12\,Q_1$ has been considered, then successive refinements have been performed up to a $96\times 96\,Q_1$ mesh. Analogous meshes has been also used for $P_1$ FE. 
Newton's method has been used with $q=4$, $\varepsilon=10^{-7}$, $\sigma=|\beta|h^4 10^{-8}$ and $\gamma=10^{-10}$. In this case, $\sigma$ has been scaled as $\vert\beta\vert^2 L^{2(d-3)}h^4$ in order to recover optimal convergence, where $L$ denotes a characteristic length of the physical domain $\Omega$. As desired, the convergence rates are not affected by the stabilization, while (as expected) the stabilized solutions have higher errors.
\begin{figure}[h]
	\centering
	\includegraphics[width=0.6\textwidth]{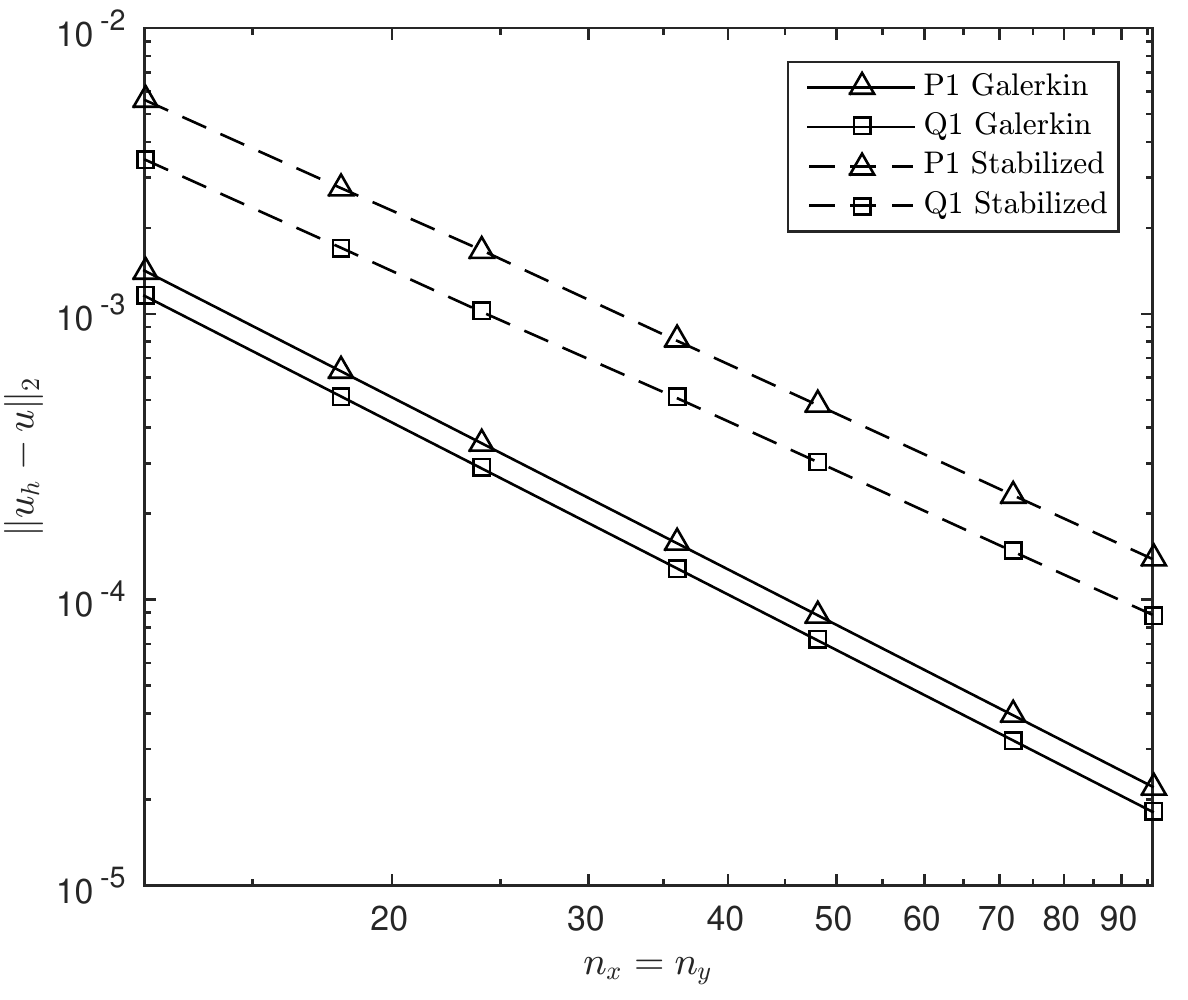}
	\caption{Convergence test, $L^2(\domain)$ error versus size of the mesh. For $P_1$ and $Q_1$ FE meshes ranging from $h=1/12$ to $h=1/96$. Newton's method has been used with parameters $q=4$, $\varepsilon=10^{-7}$, {$\sigma=|\beta|h^4 10^{-8}$} and $\gamma=10^{-10}$.}
	\label{fig:convergence}
\end{figure}

A typical linear test to assess the performance of a shock capturing method is the propagation of a discontinuity. Consider now the previous hyperbolic PDE \eqref{eq:steady_transport}
with $\mathbf{v}(x,y) \doteq \left(\nicefrac{1}{2},\, \sin\nicefrac{-\pi}{3}\right)$, and inflow boundary conditions $u_D=1$ on $\{x=0\}\cap \{y>0.7\}$ and $y=1$, while $u_D=0$ at the rest of the inflow boundary. This problem has the following analytical solution
\begin{equation}
	u(x,y) =  \left\{ \begin{array}{cl}
		1 & \text{if}\quad y > 0.7 + 2x\sin\nicefrac{-\pi}{3}, \\
		0 & \text{otherwise} .
	\end{array}\right. 
\end{equation}
At Fig. \ref{sfig:sharp-conv}, the numerical solution using the stabilization in \eqref{eq:soft_nu} is shown. A $48\times 48\,Q_1$ mesh have been used. The values chosen for the parameters in \eqref{eq:soft_nu} are $q=25$, $\varepsilon=10^{-4}$, {$\sigma=|\beta|10^{-9}$}, and $\gamma=10^{-10}$. This parameter choice makes the solution at the outflow sharp while the DMP is always satisfied. Furthermore, convergence is not jeopardized thanks to the smoothed stabilization. Particularly, it took 18 iterations for the Newton's method to converge to a nonlinear tolerance of $10^{-6}$. The non-smooth version in Fig. \ref{sfig:nonsmooth} (\eqref{eq:algebraic} with \eqref{eq:nu}) did not converge using Anderson acceleration, adding a fixed relaxation parameter of $\omega=0.5$ took 392 iterations, and 117 with Alg. \ref{alg:anderson}. In any case, observing Fig. \ref{fig:overlapedoutf}, where the outflow profile is depicted, no apparent improvement on accuracy is observed when using the non-smooth version.
\begin{figure}[h]
	\centering
	\begin{subfigure}[t]{0.45\textwidth}
		\includegraphics[width=\textwidth]{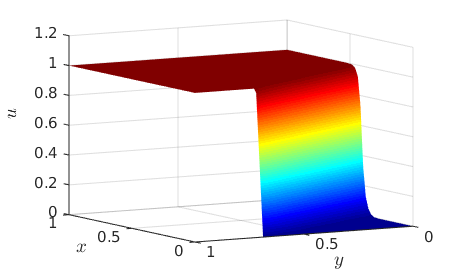}
		\caption{Smooth stabilization \eqref{eq:alpha_smooth}, with parameters $q=25$, $\varepsilon=10^{-4}$, {$\sigma=\vert\beta\vert 10^{-9}$, and $\gamma=10^{-10}$}.}
		\label{sfig:sharp-conv}
	\end{subfigure}
	\quad
	\begin{subfigure}[t]{0.45\textwidth}
		\includegraphics[width=\textwidth]{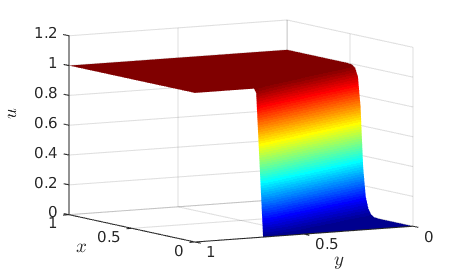}
		\caption{Non-smooth version \eqref{eq:alpha_lp} with $q=25$.}
		\label{sfig:nonsmooth}
	\end{subfigure}
	\caption{Stabilized solution of the straight propagation of a discontinuity test using the steady version of discrete problem \eqref{eq:algebraicdis} with two stabilization choices \eqref{eq:alpha_smooth} or \eqref{eq:alpha_lp}.}
\end{figure}
\begin{figure}[h]
	\centering
	\includegraphics[width=0.45\textwidth]{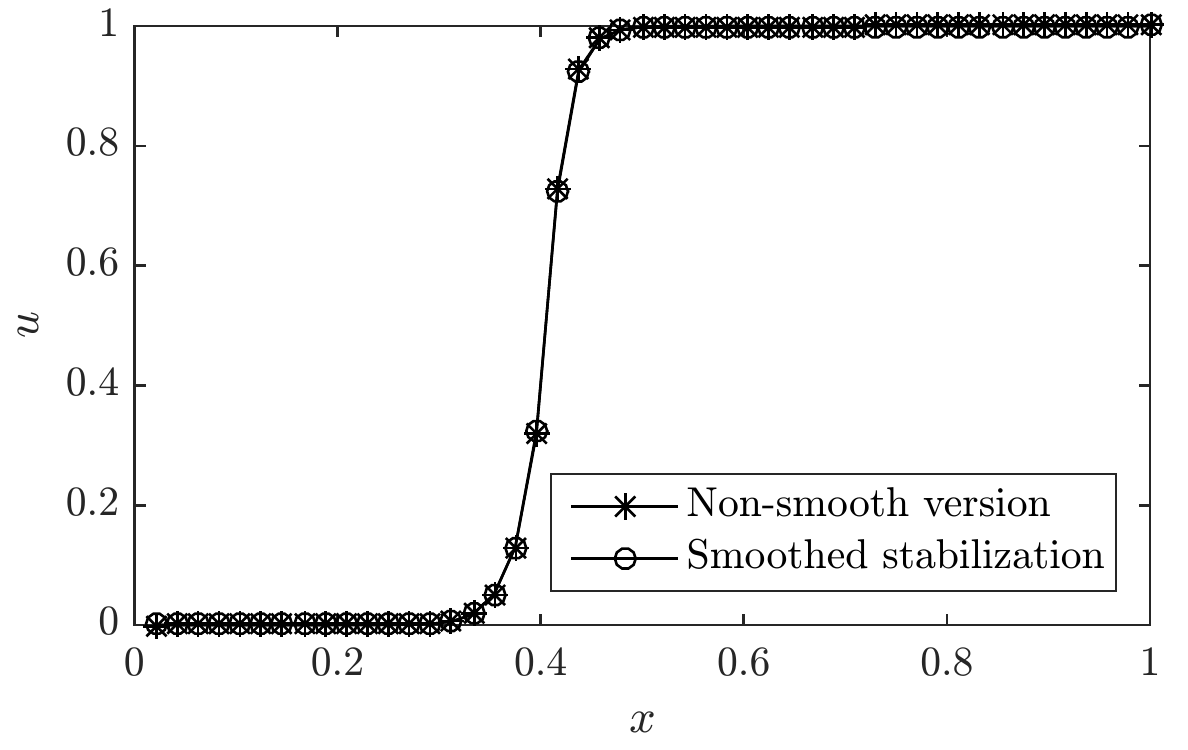}
	\caption{Stabilized solution of the straight propagation of a discontinuity test using the steady version of discrete problem \eqref{eq:algebraicdis} with two stabilization choices \eqref{eq:alpha_smooth} and \eqref{eq:alpha_lp}. The stabilization parameters used for the smoothed version are $q=25$, $\varepsilon=10^{-4}$, $\sigma=\vert\beta\vert 10^{-9}$, and $\gamma=10^{-10}$.}
	\label{fig:overlapedoutf}
\end{figure}

Fig. \ref{fig:linear-conv} shows the solution for several combinations of $q$ and $\varepsilon$, with {$\sigma=|\beta|\varepsilon 10^{-5}$} and $\gamma=10^{-10}$, solved with the two nonlinear solvers presented in the previous section over a $48\times 48 \, Q_1$ mesh. Furthermore,  the $\Vert u - \unk \Vert_{L^1}$ and $\Vert u - \unk\Vert$ errors, computed at the whole domain and restricted to the outflow boundary, are listed in Table \ref{tab:linear-conv}. These results show that, as expected, either increasing $q$ or reducing $\varepsilon$ the $L^1$ error diminishes. Nevertheless, the computational cost also increases at a higher rate. The same can be observed for the $\l2$ error. It is slightly reduced after increasing $q$ or diminishing $\varepsilon$, while this makes nonlinear convergence much harder. Moreover, comparing both nonlinear solvers in Sect. \ref{sec:solver}, it is important to note that using Newton's method the number of nonlinear iterations is reduced between 10 to 15 times.

\begin{table}[h]
	{\footnotesize
		\begin{center}
			\caption{Straight propagation test errors and iterations, using the steady version of discrete problem \eqref{eq:algebraicdis} and nonlinear diffusion \eqref{eq:soft_nu}, for different values of $q$ and $\varepsilon$, {$\sigma=|\beta|\varepsilon 10^{-5}$}, $\gamma=10^{-10}$, and both nonlinear solvers in Sect. \ref{sec:solver}.}\label{tab:linear-conv}
			\begin{tabular}{cccccccccc}\hline
				\multirow{2}{*}{$q$} & \multirow{2}{*}{$\varepsilon$} &  \multicolumn{4}{c}{Iterations} & \multirow{2}{*}{$L_1$ error} & $L_1$ error & \multirow{2}{*}{$L_2$ error} & $L_2$ error \\
				&  & A & Ap & N & Np &
				&  at $\Gamma_{\text{out}}$ & &  at $\Gamma_{\text{out}}$  \\ \hline 
				1 & $10^{-1}$ & 42 & 42 & 9 & 9 & 2.77e-02 & 5.57e-02 & 8.65e-02 & 1.23e-01 \\ 
				1 & $10^{-2}$ & 43 & 42 & 8 & 8 & 2.61e-02 & 5.16e-02 & 8.40e-02 & 1.18e-01 \\ 
				1 & $10^{-3}$ & 50 & 58 & 7 & 7 & 2.59e-02 & 5.09e-02 & 8.37e-02 & 1.17e-01 \\ 
				1 & $10^{-4}$ & 50 & 57 & 7 & 7 & 2.58e-02 & 5.08e-02 & 8.37e-02 & 1.17e-01 \\ 
				1 & $0$ & 56 & 47 &   &   & 2.59e-02 & 5.10e-02 & 8.37e-02 & 1.17e-01 \\ 
				\hdashline[0.5pt/5pt] 
				4 & $10^{-1}$ & 51 & 64 & 8 & 8 & 2.20e-02 & 4.43e-02 & 7.79e-02 & 1.12e-01 \\ 
				4 & $10^{-2}$ & 58 & 61 & 11 & 11 & 1.83e-02 & 3.45e-02 & 6.97e-02 & 9.70e-02 \\ 
				4 & $10^{-3}$ & 60 & 68 & 10 & 10 & 1.77e-02 & 3.28e-02 & 6.83e-02 & 9.44e-02 \\ 
				4 & $10^{-4}$ & 66 & 85 & 11 & 11 & 1.76e-02 & 3.25e-02 & 6.82e-02 & 9.40e-02 \\ 
				4 & $0$ & 70 & 73 &   &   & 1.76e-02 & 3.24e-02 & 6.81e-02 & 9.39e-02 \\ 
				\hdashline[0.5pt/5pt] 
				8 & $10^{-1}$ & 62 & 70 & 9 & 9 & 2.10e-02 & 4.27e-02 & 7.68e-02 & 1.11e-01 \\ 
				8 & $10^{-2}$ & 71 & 63 & 11 & 11 & 1.62e-02 & 3.04e-02 & 6.63e-02 & 9.23e-02 \\ 
				8 & $10^{-3}$ & 82 & 67 & 13 & 13 & 1.51e-02 & 2.75e-02 & 6.33e-02 & 8.74e-02 \\ 
				8 & $10^{-4}$ & 70 & 77 & 12 & 12 & 1.49e-02 & 2.69e-02 & 6.27e-02 & 8.66e-02 \\ 
				8 & $0$ & 94 & 60 &   &   & 1.48e-02 & 2.68e-02 & 6.26e-02 & 8.64e-02 \\ 
				\hdashline[0.5pt/5pt] 
				25 & $10^{-1}$ & 39 & 58 & 11 & 12 & 2.03e-02 & 4.18e-02 & 7.63e-02 & 1.11e-01 \\ 
				25 & $10^{-2}$ & 57 & 62 & 19 & 20 & 1.46e-02 & 2.78e-02 & 6.39e-02 & 8.95e-02 \\ 
				25 & $10^{-3}$ & 154 & 66 & 15 & 15 & 1.28e-02 & 2.35e-02 & 5.90e-02 & 8.24e-02 \\ 
				25 & $10^{-4}$ & 116 & 82 & 17 & 18 & 1.25e-02 & 2.27e-02 & 5.79e-02 & 8.18e-02 \\ 
				25 & $0$ & 86 & 163 &   &   & 1.23e-02 & 2.25e-02 & 5.75e-02 & 8.15e-02 \\ \hline 
				\multicolumn{10}{c}{\scriptsize A: Alg. \ref{alg:anderson} without projecting to $\fespace^{\rm adm}$, Ap: Alg. \ref{alg:anderson}.} \\
				\multicolumn{10}{c}{\scriptsize N: Alg. \ref{alg:newton} without projecting to $\fespace^{\rm adm}$, Np: Alg. \ref{alg:newton}.} 
			\end{tabular}
		\end{center}}
	\end{table}
	
	\begin{figure}[h]
		\centering
		\includegraphics[width=\textwidth]{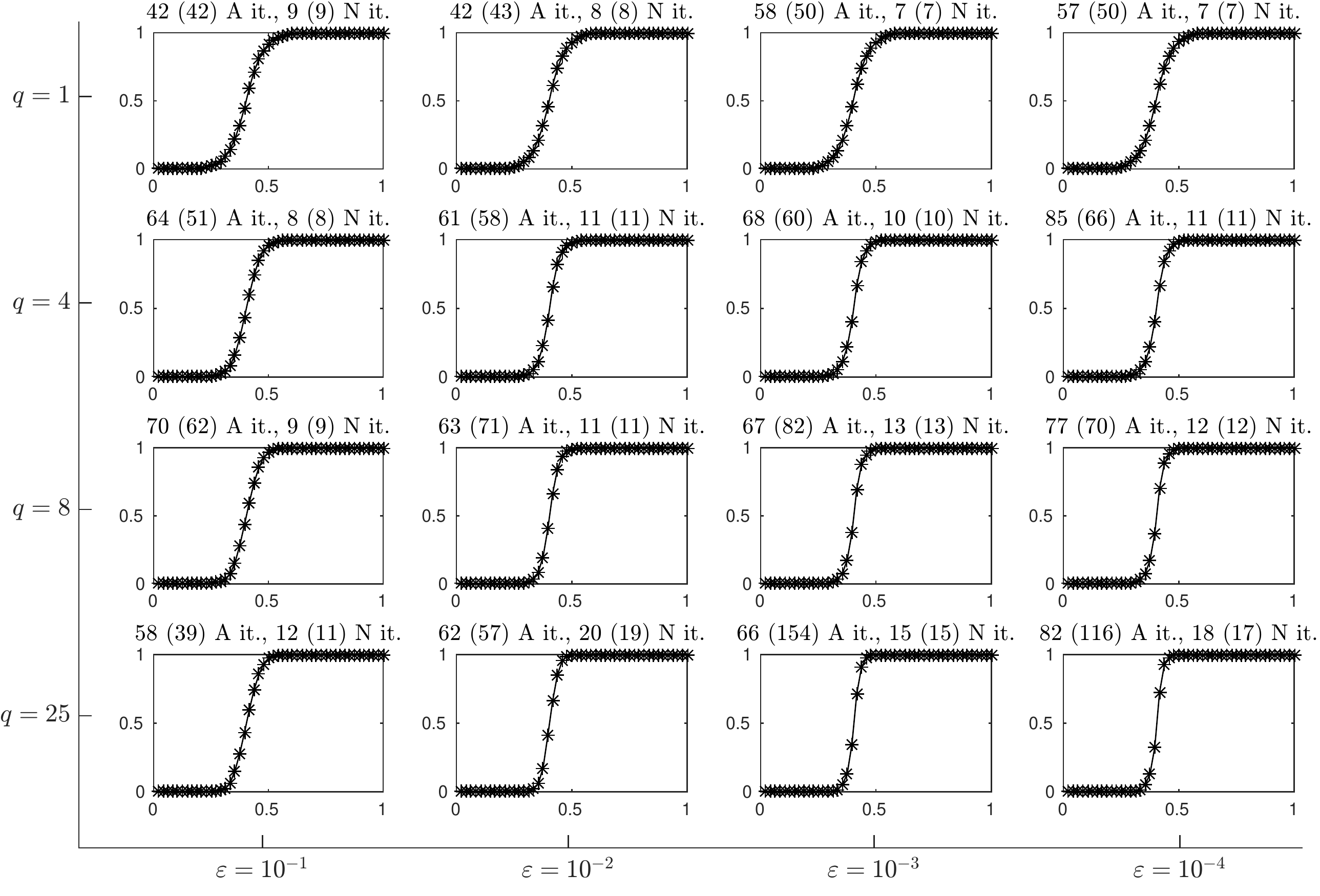}
		\caption{Straight propagation test solution at the outflow boundary $\partial\domain\backslash\inflowboundary$. Using the steady version of discrete problem \eqref{eq:algebraicdis} and nonlinear diffusion \eqref{eq:soft_nu}, for different values of $q$ and $\varepsilon$, {$\sigma=|\beta|\varepsilon 10^{-5}$}, $\gamma=10^{-10}$, and both nonlinear solvers in Sect. \ref{sec:solver}. The result in brackets shows the number of iterations if no projection to $\fespace^{\rm adm}$ is done.}
		\label{fig:linear-conv}
	\end{figure}
	
	It is important to analyze the solution at each nonlinear iteration. If the projection to the space of admissible solutions is not performed, it is possible that the solution does neither satisfy the local nor the global DMP (Def. \ref{def:lDMP} or \ref{def:gDMP}, resp.) at some nonlinear iterations. The DMP is only proved when convergence is attained. We denote by global DMP violation the difference between the global extremum of the analytical solution and the actual global extremum of the numerical solution.
	Fig. \ref{fig:dmplinear} shows the global DMP violation of the maximum and the minimum values produced at each nonlinear iteration for different values of $q$, $\varepsilon$, and $\sigma$. For $q=25$, the global DMP is clearly not satisfied at the beginning of the iterative process. In this particular case, this does not destroy the nonlinear convergence, but this is not the case in some other problems, e.g. Euler's equations. Therefore, adding a projection step to $\fespace^{\rm adm}$ is highly recommended. Further, it can be observed in Table \ref{tab:linear-conv} that in practice the projection step almost does not affect Newton convergence rate.
	\begin{figure}[h]
		\centering
		\begin{subfigure}[t]{0.45\textwidth}
			\includegraphics[width=\textwidth]{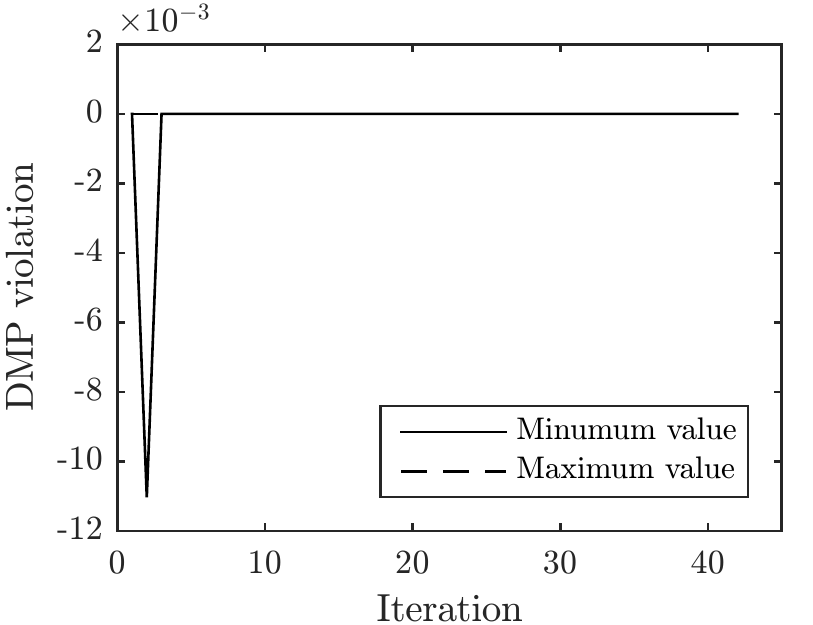}
			\caption{Using Alg. \ref{alg:anderson} with $q=1$, $\varepsilon=10^{-1}$, {$\sigma=|\beta|10^{-6}$}, and $\gamma=10^{-10}$.}
		\end{subfigure}
		\qquad
		\begin{subfigure}[t]{0.45\textwidth}
			\includegraphics[width=\textwidth]{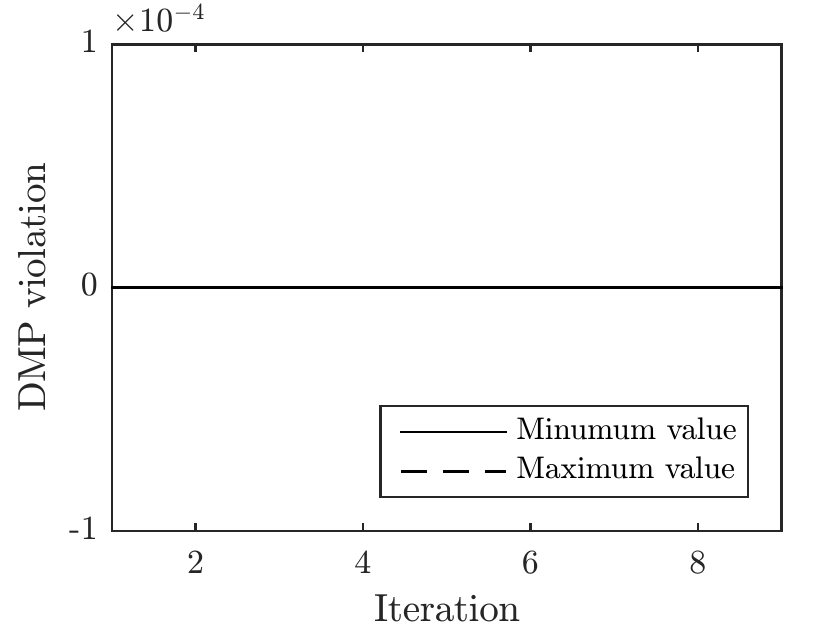}
			\caption{Using Alg. \ref{alg:newton} with $q=1$, $\varepsilon=10^{-1}$, {$\sigma=|\beta|10^{-6}$}, and $\gamma=10^{-10}$.}
		\end{subfigure}
		\begin{subfigure}[t]{0.45\textwidth}
			\includegraphics[width=\textwidth]{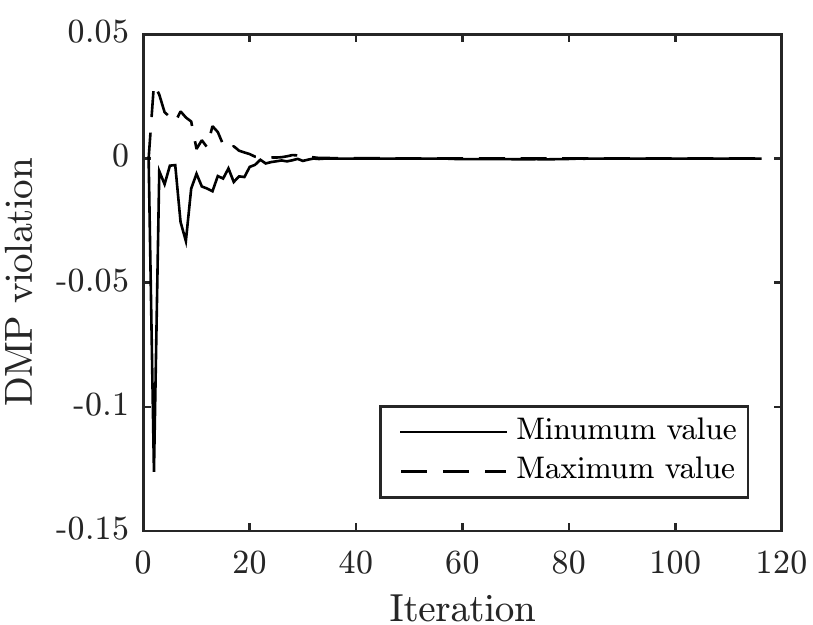}
			\caption{Using Alg. \ref{alg:anderson} with $q=25$, $\varepsilon=10^{-4}$, {$\sigma=|\beta|10^{-9}$}, and $\gamma=10^{-10}$.}
		\end{subfigure}
		\qquad
		\begin{subfigure}[t]{0.45\textwidth}
			\includegraphics[width=\textwidth]{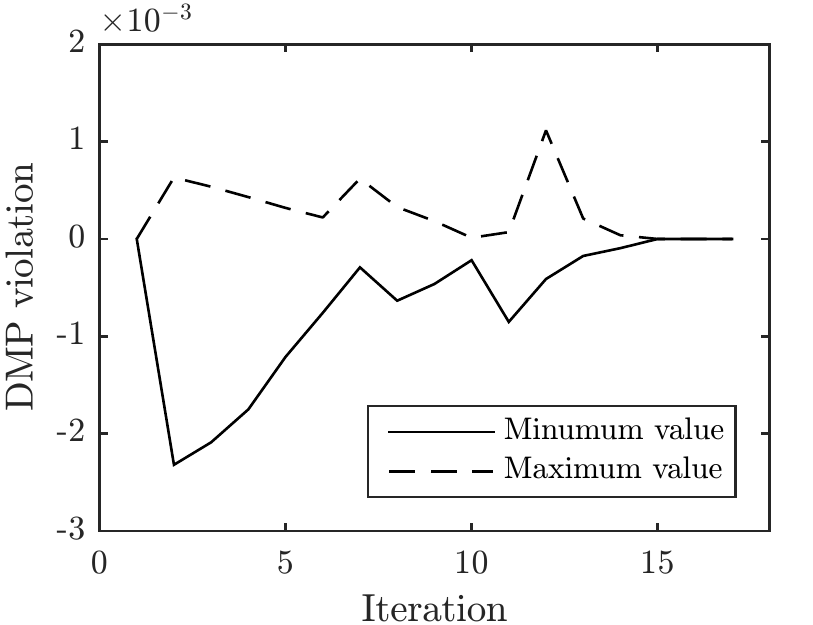}
			\caption{Using Alg. \ref{alg:newton} with $q=25$, $\varepsilon=10^{-4}$, {$\sigma=|\beta|10^{-9}$}, and $\gamma=10^{-10}$.}
		\end{subfigure}
		\caption{Evolution of global DMP violation during nonlinear iterations when avoiding the projection step in Algs. \ref{alg:anderson} and \ref{alg:newton} for the straight propagation of a discontinuity test.}
		\label{fig:dmplinear}
	\end{figure}

	Finally, it is worth to test the nonlinear convergence of the method as the mesh is refined for a problem with a discontinuity. For this purpose, we have solved the previous benchmark with $q=4$, $\varepsilon=10^{-2}$, $\sigma=\vert\beta\vert h^4 10^{-6}$, and $\gamma=10^{-10}$. The used meshes range form $12\times 12\,Q_1$ to $96\times 96\,Q_1$.

	At Fig. \ref{fig:refinement}, the number of nonlinear iterations for each mesh size is depicted. For Alg. \ref{alg:anderson} it can be observed that the number of iterations is increasing. On the contrary, this behavior is much less pronounced for Alg. \ref{alg:newton}; the number of iterations slightly increases and remains constant in the last interval.

	\begin{figure}[h]
		\centering
		\includegraphics[width=0.5\textwidth]{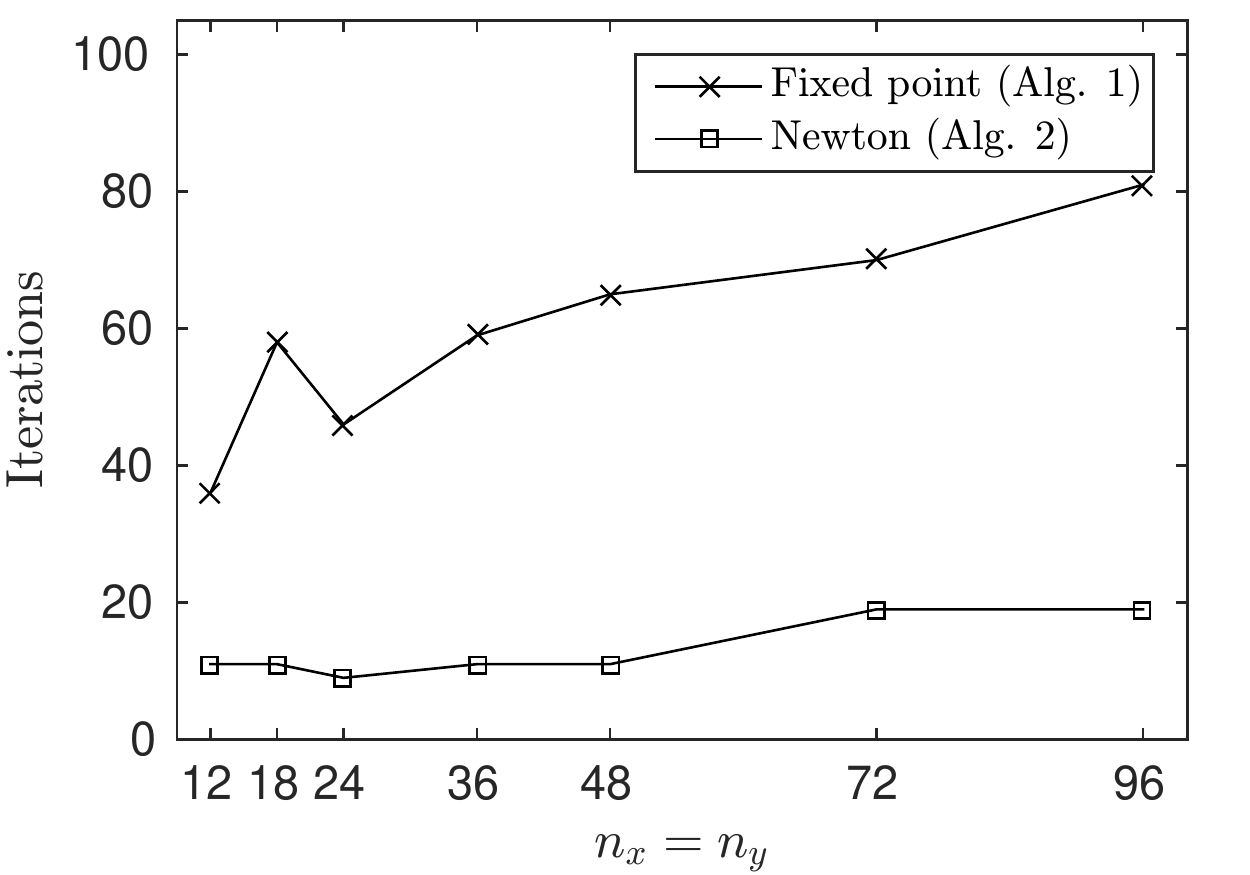}
		\caption{Straight propagation test nonlinear iterations as mesh refined from $12\times 12\,Q_1$ to $96\times 96\,Q_1$, for both Alg. \ref{alg:anderson} and Alg. \ref{alg:newton}. The shock capturing parameters used are $q=4$, $\varepsilon=10^{-2}$, $\sigma=\vert\beta\vert h^4 10^{-6}$, and $\gamma=10^{-10}$.}
		\label{fig:refinement}
	\end{figure} 
	
	Consider now the hyperbolic PDE \eqref{eq:steady_transport} on $\Omega=[0,1]\times[-1,1]$ with $\mathbf{v}(x,y) \doteq (y, -x)$, and inflow boundary conditions
	\begin{equation}
		u_D = \left\{ \begin{array}{cl}
			1 & \text{if}\quad 0.35 < x < 0.65, \\
			0 & \text{otherwise}.
		\end{array}\right. 
	\end{equation}
	This particular configuration has the following analytical solution
	\begin{equation}
		u(x,y) =  \left\{ \begin{array}{cl}
			1 & \text{if}\quad 0.35 < \sqrt{x^2 + y^2} < 0.65 ,\\
			0 & \text{otherwise}.
		\end{array}\right. 
	\end{equation}
	At Fig. \ref{fig:circular-conv} the solutions at the outflow boundary are depicted for several combinations of $q$ and $\varepsilon$, with {$\sigma=|\beta|\varepsilon 10^{-5}$} and $\gamma=10^{-10}$. In all cases, we have considered the two schemes presented in  Sect. \ref{sec:solver} using a $64\times 128\,Q_1$ FE mesh. As for the previous numerical experiment, we collect the number of iterations and the errors in Table \ref{tab:circular-conv}. We observe that it is particularly difficult to converge to the solution for $q=1$ and small values of $\varepsilon$. In any case, for $q$ equal to 4 or greater, the number of iterations increase with $q$, as naturally expected. We also observe in this test that the number of nonlinear iterations can be highly reduced using Newton's method. Particularly, it reduces the number of nonlinear iterations up to 20 times. 3D plots of the smoothest and the sharpest solutions in Fig. \ref{fig:circular-conv}  (respectively top-left and bottom-right subfigures) are shown in Fig. \ref{fig:circular-conv-3D} .
	
	\begin{table}
		{\footnotesize
			\begin{center}
				\caption{Circular propagation test errors and iterations, using the steady version of discrete problem \eqref{eq:algebraicdis} and nonlinear diffusion \eqref{eq:soft_nu}, for different values of $q$ and $\varepsilon$, {$\sigma=|\beta|\varepsilon 10^{-5}$}, $\gamma=10^{-10}$, and both nonlinear solvers in Sect. \ref{sec:solver}.}\label{tab:circular-conv}
				\begin{tabular}{cccccccccc}\hline
					\multirow{2}{*}{$q$} & \multirow{2}{*}{$\varepsilon$} &  \multicolumn{4}{c}{Iterations} & \multirow{2}{*}{$L_1$ error} & $L_1$ error & \multirow{2}{*}{$L_2$ error} & $L_2$ error \\
					&  & A & Ap & N & Np &
					&  at $\Gamma_{\text{out}}$ & &  at $\Gamma_{\text{out}}$  \\ \hline
					1 & $10^{-1}$ & 30 & 30 & 9 & 9 & 1.42e-01 & 1.93e-01 & 2.01e-01 & 2.36e-01 \\ 
					1 & $10^{-2}$ & -- & 54 & 10 & 10 & 1.11e-01 & 1.50e-01 & 1.74e-01 & 2.05e-01 \\ 
					1 & $10^{-3}$ & -- & -- & 11 & 11 & 1.05e-01 & 1.42e-01 & 1.68e-01 & 1.99e-01 \\ 
					1 & $10^{-4}$ & 196 & -- & 19 & 19 & 1.04e-01 & 1.40e-01 & 1.68e-01 & 1.98e-01 \\ 
					1 & $0$ & -- & -- &   &   & -- & -- & -- & -- \\ 
					\hdashline[0.5pt/5pt] 
					4 & $10^{-1}$ & 23 & 23 & 10 & 10 & 1.33e-01 & 1.82e-01 & 1.97e-01 & 2.31e-01 \\ 
					4 & $10^{-2}$ & 64 & 64 & 15 & 15 & 8.47e-02 & 1.15e-01 & 1.55e-01 & 1.84e-01 \\ 
					4 & $10^{-3}$ & 105 & 111 & 22 & 22 & 6.74e-02 & 9.31e-02 & 1.34e-01 & 1.64e-01 \\ 
					4 & $10^{-4}$ & -- & 139 & 24 & 24 & 6.38e-02 & 8.88e-02 & 1.29e-01 & 1.60e-01 \\ 
					4 & $0$ & 198 & 194 &   &   & 6.31e-02 & 8.80e-02 & 1.28e-01 & 1.59e-01 \\ 
					\hdashline[0.5pt/5pt] 
					8 & $10^{-1}$ & 23 & 22 & 11 & 11 & 1.32e-01 & 1.81e-01 & 1.97e-01 & 2.31e-01 \\ 
					8 & $10^{-2}$ & 73 & 68 & 15 & 15 & 8.10e-02 & 1.10e-01 & 1.53e-01 & 1.82e-01 \\ 
					8 & $10^{-3}$ & 95 & 96 & 19 & 19 & 5.91e-02 & 8.18e-02 & 1.28e-01 & 1.57e-01 \\ 
					8 & $10^{-4}$ & 100 & 109 & 22 & 22 & 5.28e-02 & 7.46e-02 & 1.18e-01 & 1.50e-01 \\ 
					8 & $0$ & 256 & 231 &   &   & 5.12e-02 & 7.28e-02 & 1.16e-01 & 1.48e-01 \\ 
					\hdashline[0.5pt/5pt] 
					25 & $10^{-1}$ & 22 & 22 & 14 & 14 & 1.32e-01 & 1.80e-01 & 1.97e-01 & 2.31e-01 \\ 
					25 & $10^{-2}$ & 45 & 49 & 16 & 15 & 7.82e-02 & 1.07e-01 & 1.51e-01 & 1.80e-01 \\ 
					25 & $10^{-3}$ & 77 & 70 & 20 & 20 & 5.37e-02 & 7.50e-02 & 1.24e-01 & 1.54e-01 \\ 
					25 & $10^{-4}$ & 131 & 109 & 23 & 24 & 4.51e-02 & 6.49e-02 & 1.11e-01 & 1.44e-01 \\ 
					25 & $0$ & 180 & 289 &   &   & 4.22e-02 & 6.14e-02 & 1.06e-01 & 1.39e-01 \\   \hline
					\multicolumn{10}{c}{\scriptsize A: Alg. \ref{alg:anderson} without projecting to $\fespace^{\rm adm}$, Ap: Alg. \ref{alg:anderson}.} \\
					\multicolumn{10}{c}{\scriptsize N: Alg. \ref{alg:newton} without projecting to $\fespace^{\rm adm}$, Np: Alg. \ref{alg:newton}.} 
				\end{tabular}
			\end{center}}
		\end{table}

		\begin{figure}[h]
			\centering
			\includegraphics[width=\textwidth]{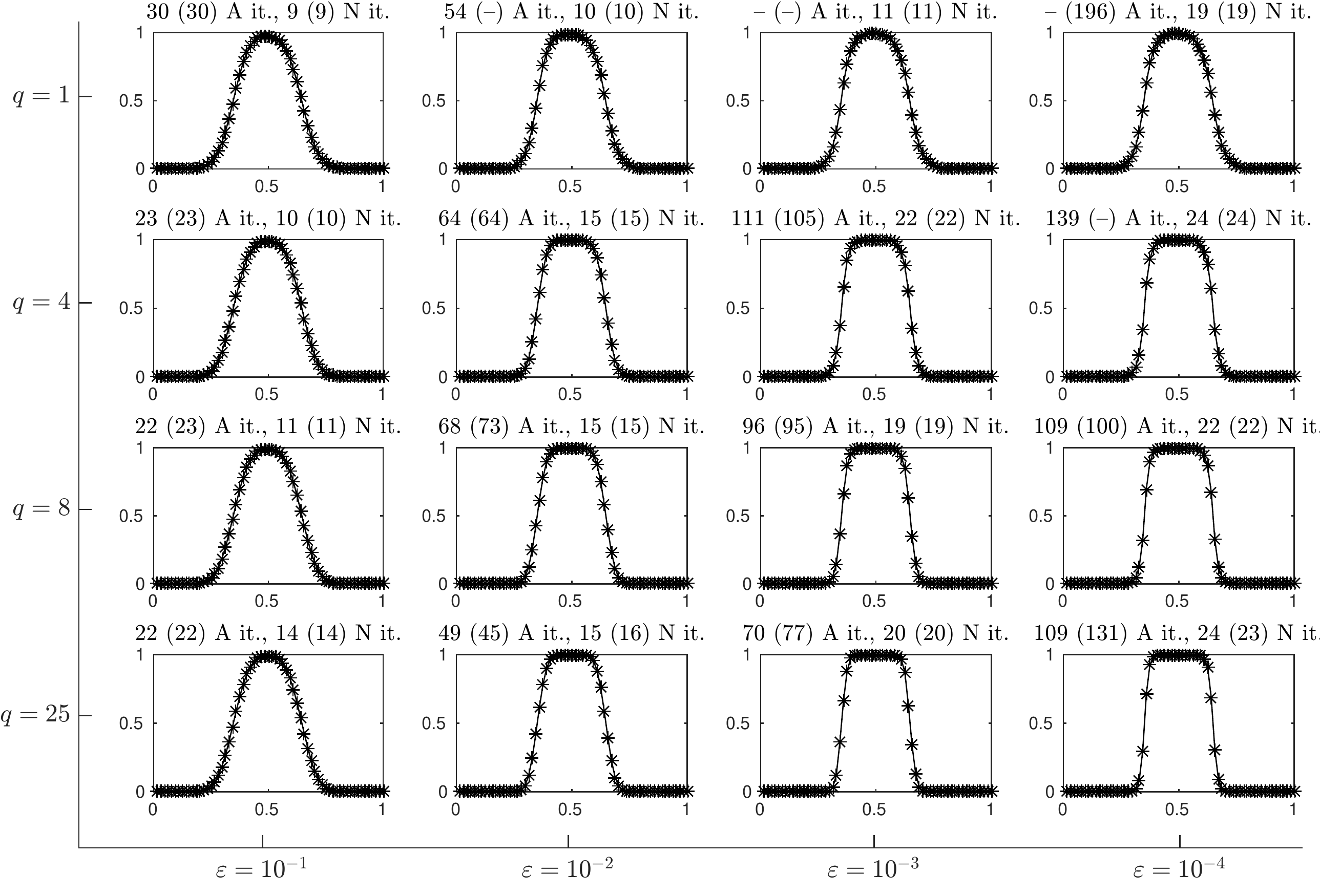}
			\caption{Circular propagation test solution at the outflow boundary $\partial\domain\backslash\inflowboundary$. Using the steady version of discrete problem \eqref{eq:algebraicdis} and nonlinear diffusion \eqref{eq:soft_nu}, for different values of $q$ and $\varepsilon$, {$\sigma=|\beta|\varepsilon 10^{-5}$}, $\gamma=10^{-10}$ and both nonlinear solvers in Sect. \ref{sec:solver}. The result in brackets shows the number of iterations if no projection to $\fespace^{\rm adm}$ is done.}
			\label{fig:circular-conv}
		\end{figure}
		
		\begin{figure}[h]
			\centering  
			\begin{subfigure}[t]{0.45\textwidth}
				\includegraphics[width=\textwidth]{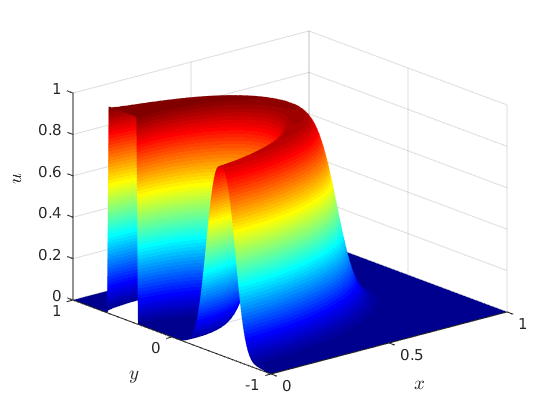}
				\caption{Smoothest solution with parameters: $q=1$, $\varepsilon=10^{-1}$, {$\sigma=|\beta|10^{-6}$}, and $\gamma=10^{-10}$.}
			\end{subfigure}
			\quad
			\begin{subfigure}[t]{0.45\textwidth}
				\includegraphics[width=\textwidth]{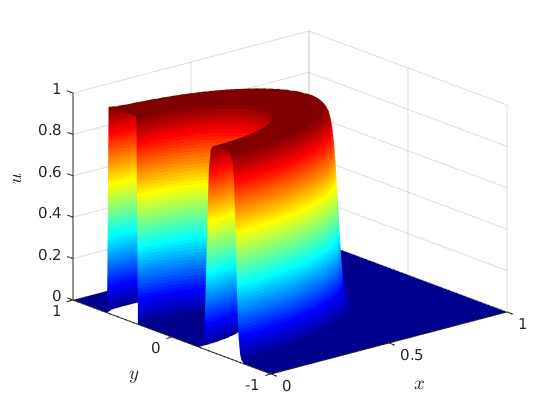}
				\caption{Sharpest solution with parameters: $q=25$ and $\varepsilon=10^{-4}$, {$\sigma=|\beta|10^{-9}$}, and $\gamma=10^{-10}$.}
			\end{subfigure}
			\caption{Stabilized solution of the circular convection test using the steady version of the discrete problem \eqref{eq:algebraicdis} and the nonlinear diffusion \eqref{eq:soft_nu} for two different parameter choices.}
			\label{fig:circular-conv-3D}
		\end{figure}
		
		Fig. \ref{fig:dmpcircular} shows that in this second test, as in the previous one, if the projection step is not performed the global DMP (Def. \ref{def:gDMP}) is not satisfied at all nonlinear iterations. This is specially evident for the combination shown in the figure, i.e., high values of $q$ and low values of $\varepsilon$ and $\sigma$.
		\begin{figure}[h]
			\centering
			\begin{subfigure}[t]{0.45\textwidth}
				\includegraphics[width=\textwidth]{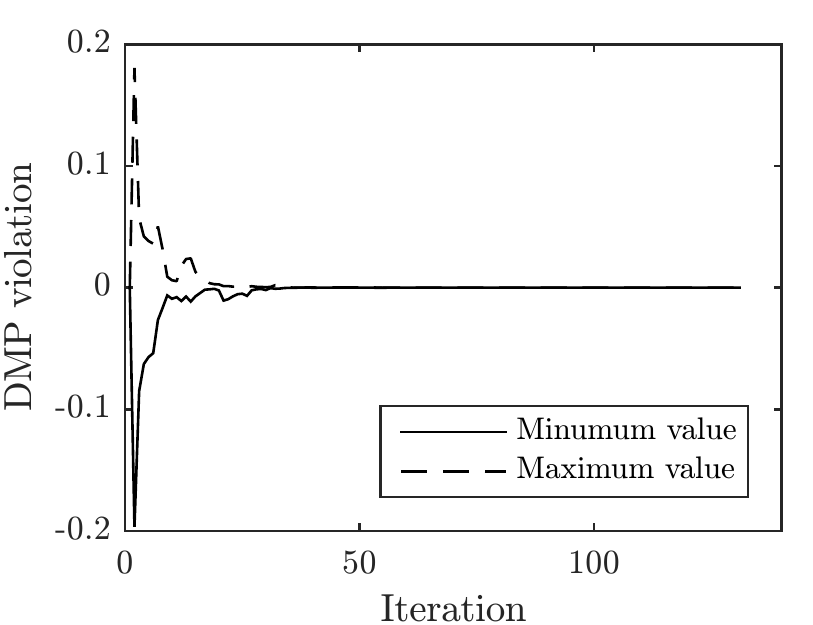}
				\caption{Using Alg. \ref{alg:anderson}}
			\end{subfigure}
			\begin{subfigure}[t]{0.45\textwidth}
				\includegraphics[width=\textwidth]{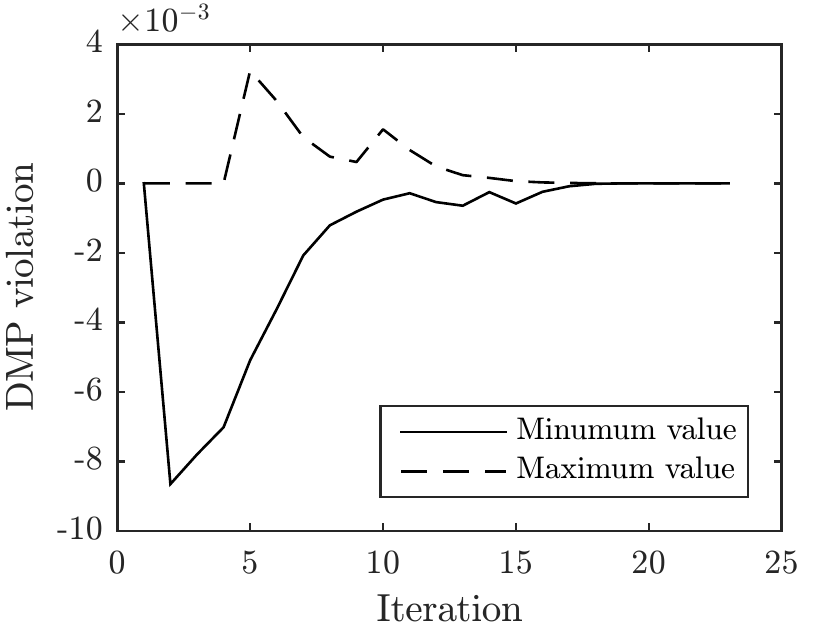}
				\caption{Using Alg. \ref{alg:newton}}
			\end{subfigure}
			\caption{Evolution of global DMP violation during nonlinear iterations when avoiding the projection step in Algs. \ref{alg:anderson} and \ref{alg:newton}  for the circular propagation of a discontinuity. Using $q=25$, $\varepsilon=10^{-4}$, {$\sigma=|\beta|10^{-9}$}, $\gamma=10^{-10}$.}
			\label{fig:dmpcircular}
		\end{figure}
		
		\subsection{Transient transport problems}
		Let us test the performance of the stabilization method in Sect. \ref{sec:dif_stab} for transient problems. For this purpose we will consider the 3 body rotation benchmark that reads as: 
		\begin{equation}\label{eq:transient_transport}
			\begin{array}{rlcl}
				\partial_t u + \gradient\cdot(\mathbf{v}u) &= 0 &\text{ in }& \domain=[0,1]\times[0,1] ,\\
				u &= 0 &\text{ on } & \inflowboundary ,\\
				u &= u_0 &\text{ at } & t=0 ,
			\end{array}
		\end{equation}
		where $\mathbf{v}\doteq(\nicefrac{1}{2}-y,x-\nicefrac{1}{2})$ and
		\begin{equation}
			u_0(x,y)\doteq\left\{\begin{array}{ccl}
				\frac{1}{4} +\nicefrac{ \cos\left(\frac{\pi\sqrt{(x-0.25)^2 + (y-0.5)^2}}{0.15}\right)}{4} & \text{if} & \nicefrac{\sqrt{(x-0.25)^2 + (y-0.5)^2}}{0.15} \leq 1 \\
				1- \nicefrac{\sqrt{(x-0.5)^2 + (y-0.25)^2}}{0.15} & \text{if} & \nicefrac{\sqrt{(x-0.5)^2 + (y-0.25)^2}}{0.15} \leq 1 \\
				1 & \text{if} &\left\{\begin{array}{c}
					\nicefrac{\sqrt{(x-0.5)^2 + (y-0.75)^2}}{0.15} \leq 1 \\ 
					0.55<x<0.45,\,y>0.85 
				\end{array}\right.
			\end{array}
			\right. .
		\end{equation}
		
		The above problem is solved in a $150\times 150\ Q_1$ FE mesh, with solver parameters $q=25$, $\gamma=10^{-8}$, {$\sigma=|\beta|10^{-10}$}, and $\varepsilon=10^{-4}$. The discretization in time is performed using the BE method with a time step of $10^{-3}$. 
		At Fig. \ref{fig:bodyrotationIC}, the initial solution is depicted. Figs. \ref{sfig:bodyrota} to \ref{sfig:bodyrotc} show the solution after one revolution (at time $t=2\pi$). 
		
		The solution obtained with the stabilization in  \eqref{eq:soft_nu}, \eqref{eq:nuConsMass}, and \eqref{eq:nu} are depicted in Figs. \ref{sfig:bodyrota}, \ref{sfig:bodyrotb}, and \ref{sfig:bodyrotc}, respectively. It is observed that the symmetric mass matrix method yields slightly more diffusive solutions than the LED method. This can be better observed in Fig. \ref{fig:bodyrotationCS}, where a cross-section of each of the figures rotated is depicted at $t=0$ and after one revolution ($t=2\pi$) for all three methods. As naturally expected, regularizing the stabilization makes the method faster to converge but the solution becomes smoother. Nevertheless, the regularization parameters ($\sigma$ and $\varepsilon$) allow one to take the choice that better fits the requirements, either a faster but smoother method or the opposite. In any case, all schemes satisfy the DMP at all time steps.
		
		\begin{figure}[h]
			\centering
			\begin{subfigure}[t]{0.45\textwidth}
				\includegraphics[width=\textwidth]{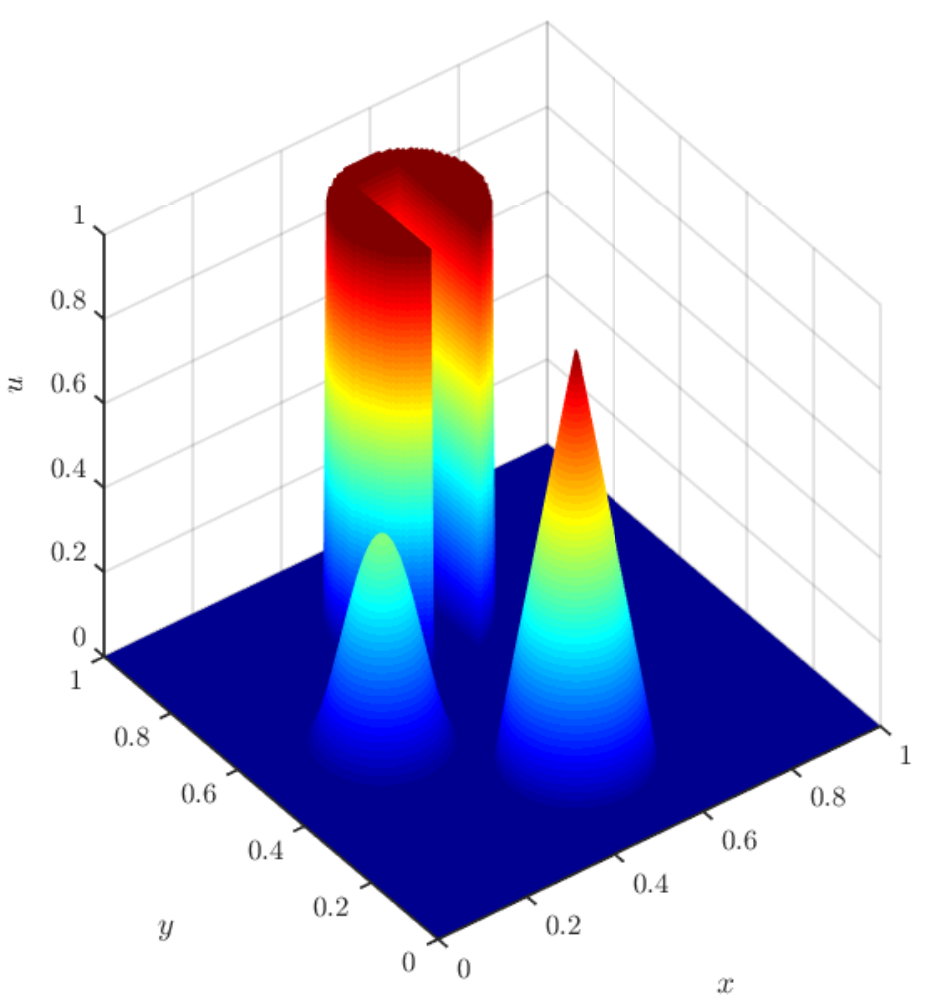}
				\caption{\footnotesize Initial conditions.}
				\label{fig:bodyrotationIC}	
			\end{subfigure}
			\begin{subfigure}[t]{0.45\textwidth}
				\includegraphics[width=\textwidth]{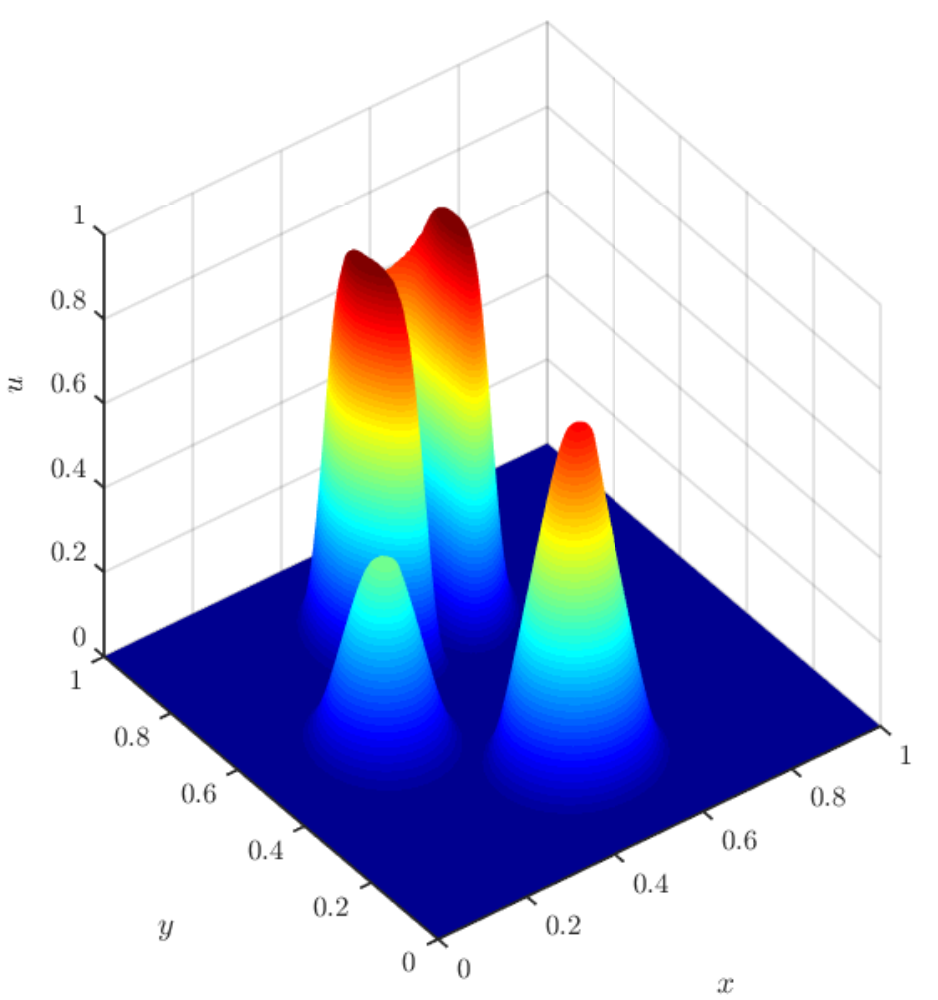}
				\caption{\footnotesize LED scheme.}\label{sfig:bodyrota}
			\end{subfigure}
			\begin{subfigure}[t]{0.45\textwidth}
				\includegraphics[width=\textwidth]{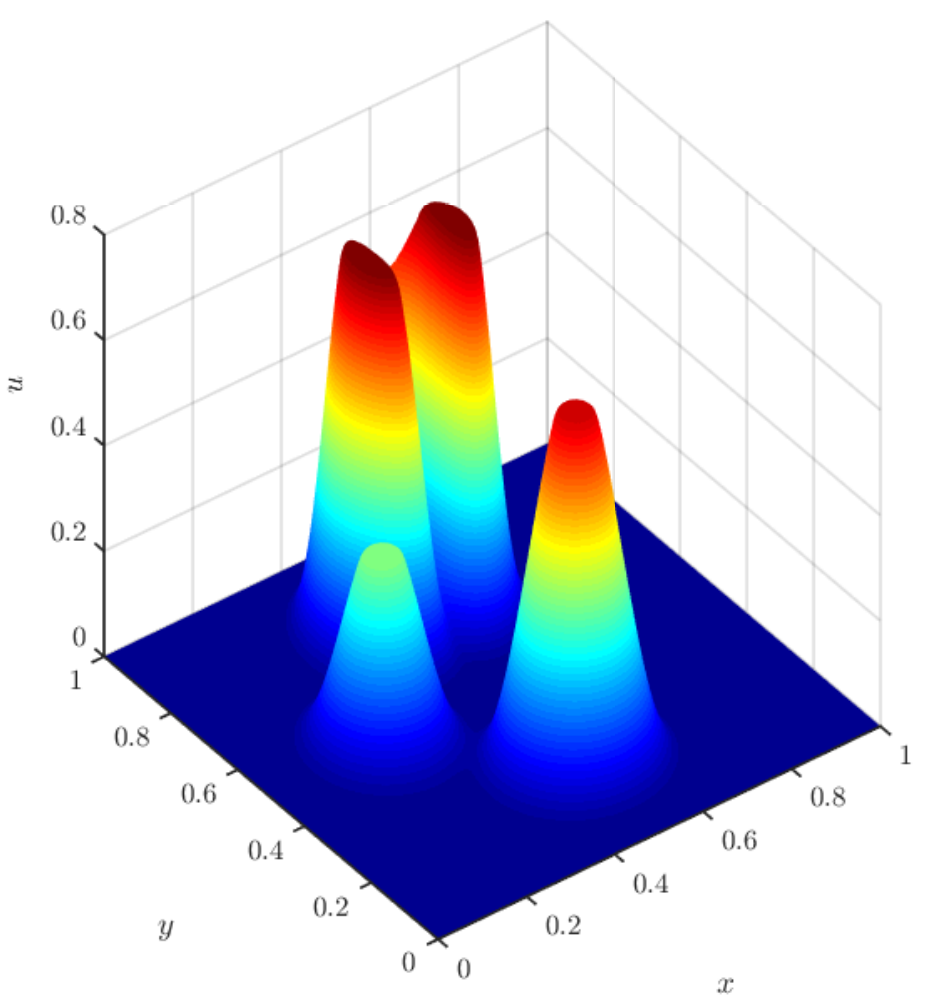}
				\caption{\footnotesize Global DMP scheme.}\label{sfig:bodyrotb}
			\end{subfigure}
			\begin{subfigure}[t]{0.45\textwidth}
				\includegraphics[width=\textwidth]{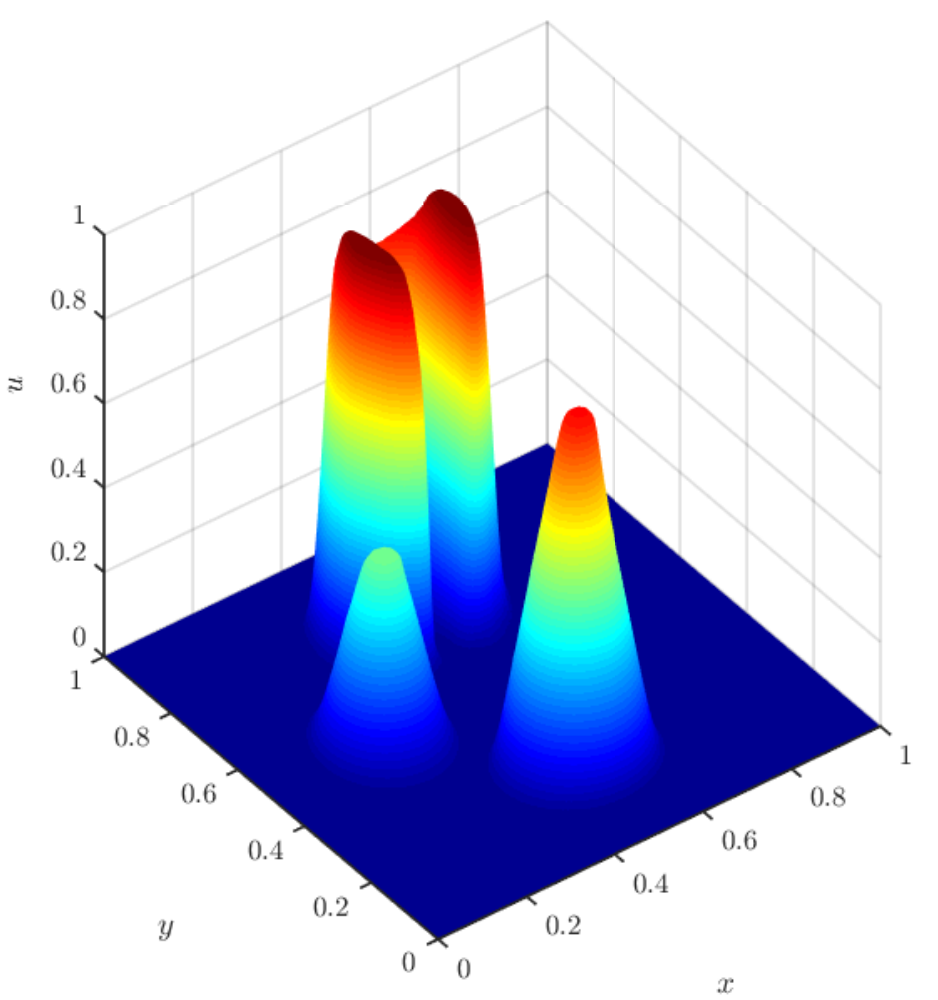}
				\caption{\footnotesize LED DMP nonsmooth stabilization.}\label{sfig:bodyrotc}
			\end{subfigure}
			\caption{3 Body rotation test results using discrete problem \eqref{eq:algebraicdis} and two different artificial diffusions (\eqref{eq:soft_nu} leading an LED scheme, and \eqref{eq:nuConsMass} with \eqref{eq:alpha_smooth} leading a global DMP scheme). Using a $150\times 150\ Q_1$ element mesh, and parameters: $q=25$, $\gamma=10^{-8}$, {$\sigma=|\beta|10^{-10}$}, $\varepsilon=10^{-4}$, and $\Delta t=10^{-3}$.}
			\label{fig:bodyrotation}
		\end{figure}
		\begin{figure}[h]
			\centering
			\includegraphics[width=0.9\textwidth]{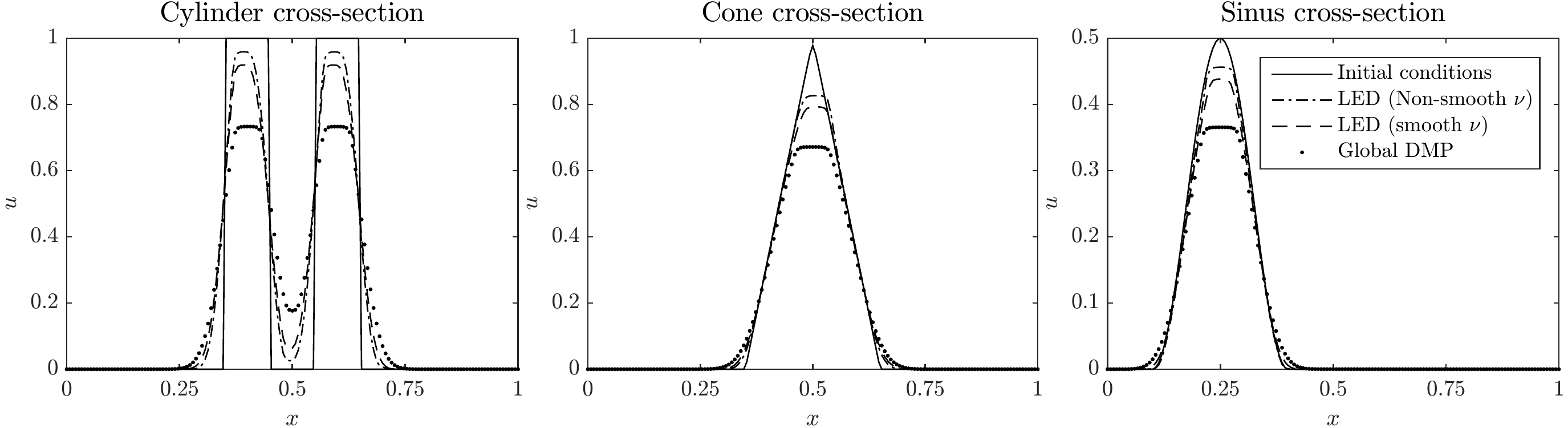}
			\caption{Cross-sections of each for the figures rotated in the three body rotation benchmark. The parameters used are $q=25$, $\gamma=10^{-8}$, $\sigma=|\beta|10^{-10}$, $\varepsilon=10^{-4}$, and $\Delta t=10^{-3}$, in a $150\times 150\ Q_1$ element mesh. The discrete problem \eqref{eq:algebraicdis} is used in combination with three different artificial diffusions \eqref{eq:soft_nu} and \eqref{eq:nu} leading to a LED scheme, and \eqref{eq:nuConsMass} leading to a global DMP scheme.}
			\label{fig:bodyrotationCS}
		\end{figure}
		
		\subsection{Burgers' equation}
		Finally, let us test our stabilization with a nonlinear transient problem. Particularly the 2D Burgers' equation, i.e. equation \eqref{eq:transient_transport} with $\mathbf{v}\doteq(1,1)\nicefrac{u}{2}$, is solved on $\domain=[0,1]\times[0,1]$ using a $150\times 150\,Q_1$ mesh. 
		The discretization in time is performed using the BE method with a time step of $10^{-2}$. The initial conditions at $t=0$ are
		\begin{equation}
			u_0(x,y)\doteq \left\{\begin{array}{ccccc}
				-0.2 & \text{if}& x<0.5 &\text{and} & y>0.5 \\
				-1 & \text{if}& x>0.5 &\text{and} & y>0.5 \\
				0.5 & \text{if}& x<0.5 &\text{and} & y<0.5 \\
				0.8 & \text{if}& x>0.5 &\text{and} & y<0.5
			\end{array}
			\right. ,
		\end{equation}
		and the solution is advanced until $t=0.5$.
		
		\begin{figure}[h]
			\centering
			\begin{subfigure}[b]{0.45\textwidth}
				\includegraphics[width=\textwidth]{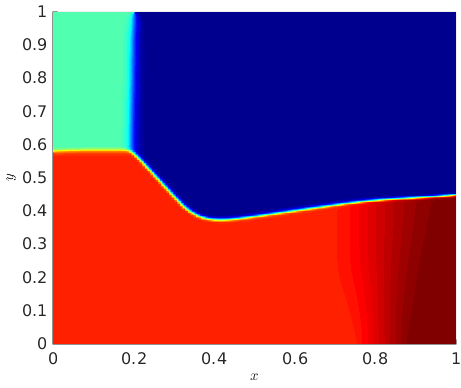}
				\caption{Solution for:$q=1$, $\varepsilon=10^{-2}$, {$\sigma=|\beta|10^{-6}$}, and $\gamma=10^{-8}$.}\label{sfig:burgersa}
			\end{subfigure}
			\quad
			\begin{subfigure}[b]{0.45\textwidth}
				\includegraphics[width=\textwidth]{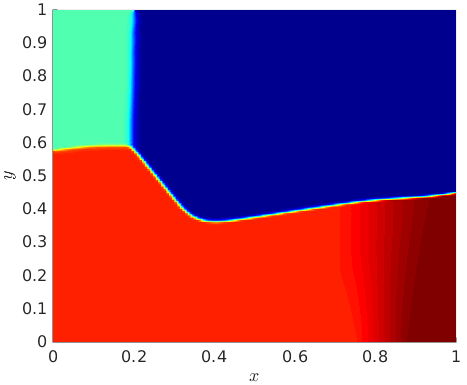}
				\caption{Solution for: $q=4$, $\varepsilon=10^{-3}$, {$\sigma=|\beta|10^{-7}$}, and $\gamma=10^{-8}$.}\label{sfig:burgersb}
			\end{subfigure}
			\caption{Burger's equation solutions at $t=0.5$ using discrete problem \eqref{eq:algebraicdis} and \eqref{eq:nu} with \eqref{eq:soft_nu}. Using a $150\times 150\ Q_1$ element mesh, $\Delta t=10^{-2}$, and two sets of parameters $q$, $\gamma$, $\sigma$, and $\varepsilon$.}
			\label{fig:burgers}
		\end{figure}
		
		The following stabilization parameters have been used for obtaining the results in Fig. \ref{sfig:burgersa}: $q=1$, $\varepsilon=10^{-3}$, {$\sigma=|\beta|10^{-6}$}, and $\gamma=10^{-8}$.
		Although the parameters used are not enforcing a particularly sharp solution (see Figs. \ref{fig:linear-conv} and \ref{fig:circular-conv}), Fig. \ref{sfig:burgersa} shows properly transported and minimally smeared shocks. Only in the lower right region the method appears to be more diffusive than desired. Notice that in that region the gradient in the $x$ direction spreads as $y$ increases, while it should not. Nevertheless, in Fig. \ref{sfig:burgersb}, that shows the solution for $q=4$, $\varepsilon=10^{-4}$, {$\sigma=|\beta|10^{-7}$}, and $\gamma=10^{-8}$. the method is less diffusive and the obtained shocks are even sharper. In any case, both choices satisfy the DMP for all time steps.
		
		\section{Conclusions}\label{sec:conclusions}
		
		In this work, we have considered a nonlinear stabilization technique for the FE approximation of scalar conservation laws with implicit time stepping. The method relies on an artificial diffusion method, based on a graph-Laplacian operator. The artificial diffusion is judiciously chosen in order to satisfy a local DMP for steady problems. It is nonlinear, since it depends on a shock detector. Further, the resulting method is linearity preserving. The same shock detector is used to gradually lump the mass matrix. The resulting method is LED, positivity preserving, and also satisfies a global DMP. Lipschitz continuity has also been proved.
		
		However, the resulting scheme is highly nonlinear, leading to very poor nonlinear convergence rates, even using Anderson acceleration techniques. It is due to the fact that the nonlinear operator to be inverted at every time step is non-differentiable. The critical problem of nonlinear convergence of implicit monotonic methods based on nonlinear artificial diffusion have already been previously reported in the literature (see \cite{kuzmin_linearity-preserving_2012}). As a result, we propose a smooth version of the scheme. It leads to twice differentiable nonlinear stabilization schemes, which allows one to straightforwardly use Newton's method using the exact Jacobian. Twice differentiability ensures quadratic convergence.
		
		We have considered two nonlinear solvers, namely Anderson acceleration and Newton's method. We have observed numerically that the effect of the smoothness has a positive impact in the reduction of the computational cost. The impact of using Newton's method versus Anderson acceleration is also very positive. In general, using the Newton method with a smooth version of the method we can reduce 10 to 20 times the number of iterations of Anderson acceleration with the original non-smooth algorithms. 
		
		All the monotonic properties are satisfied (as theoretically proved) in the numerical experiments. Steady and transient linear transport, and transient Burgers' equation have been considered in 2D. In any case, these properties are only true for the converged solution, but not for iterates. In this sense, we have also proposed the concept of projected nonlinear solvers, where a projection step is performed at the end of every nonlinear iterations onto a FE space of admissible solutions. The space of admissible solutions is the one that satisfies the desired monotonic properties (maximum principle or positivity). The projection has no effect on the quality of the nonlinear convergence.
		
		Future work should tackle the entropy stability analysis of the resulting schemes when applied to nonlinear problems. Some initial results in this direction can be found in \cite{burman_nonlinear_2007}. The extension to systems of conservation laws and higher order methods in space and time is another interesting line of research.

		\appendix
		\section{Proof of Theorem \ref{th:lipschitz}}\label{proof_continuity}
		
		Let us proof Theorem \ref{th:lipschitz}. We assume that the FE mesh is quasi-uniform in order to reduce technicalities. However, the proof for Lipschitz continuity can be extended to more general meshes. We denote $A=cB$ as $A\eqsim B$ and $A<cB$ as $A\lesssim B$, for any positive constant $c$ that does not depend on the numerical or physical parameters.
		
		From the definition of the nonlinear stabilization in \eqref{eq:stab}, we get
		\begin{align}
			| \langle \B(u)u,z \rangle - \langle \B(v)v,z \rangle| \leq  & 
			\left|\sum_{{i \in \nodes }}
			\sum_{{j\in\neighborhood[i]}} \nu_\ij(v)  \ell(i,j) (u_j-v_j) z_i \right|
			\label{res1} \\ &  
			+ 
			\left| 
			\sum_{{i \in \nodes }}
			\sum_{{j\in\neighborhood[i]}} (\nu_\ij(u) - \nu_\ij(v)) \ell(i,j) u_j z_i\right|.
		\end{align}
		Using the definition of $|\beta|$, the Cauchy-Schwarz inequality, the fact that $\| \shapef[i ] \| \leq C h^{d/2}$, and the inverse inequality $\| \nabla v_h \| \lesssim h^{-1} \| v_h \|$ for $v_h \in \fespace$ (see \cite{BrennerScott08}), we get:
		\begin{equation}\label{res2}
			\F_\ij(w) \leq | \beta | \| \gradient \shapef[i] \|_\l2 \| \shapef[j] \|_\l2 \lesssim h^{d-1} | \beta |,
		\end{equation}
		for any $w \in \fespace^{\rm adm}$. Using \eqref{res2}, the first term in the RHS of \eqref{res1} is bounded as follows:
		$$
		\left| \sum_{{i \in \nodes }}
		\sum_{{j\in\neighborhood[i]}} \nu_\ij(v) \ell(i,j)(u_j-v_j)z_i \right| \lesssim
		h^{d-1}| \beta | | u - v |_\ell | z |_\ell.
		$$
		The second term is bounded using the Cauchy-Schwarz inequality:
		\begin{align}
			& \sum_{{i \in \nodes }} \sum_{{j\in\neighborhood[i]}} (\nu_\ij(u) -
			\nu_\ij(v)) \ell(i,j) u_j z_i \label{boundij3}
			\lesssim
			\left | \sum_{{i \in \nodes }} \sum_{{j\in\neighborhood[i]}} \half (\nu_\ij(u) -
			\nu_\ij(v))^2 (u_i - u_j)^2 \right |^\half \times | z |_\ell.
		\end{align}
		Using \eqref{res2}, we have:
		\begin{align}
			& \nu_\ij(u) - \nu_\ij(v)   \\ & = \max\{\shock[i](u) \F_{\ij}(u),\shock[j](u) \F_{\ji}(u),0\} - \max\{\shock[i](v) \F_{\ij}(v),\shock[j](v) \F_{\ji}(v),0\} \label{viscobound}\\
			&\leq \max\{(\shock[i](u)\F_{\ij}(u)  - \shock[i](v) \F_{\ij}(v), \shock[j](u) \F_{\ji}(u) - \shock[j](v) \F_{\ji}(v),0\} \\
			& \lesssim h^{d-1}| \beta | \max\{|\shock[i](u) - \shock[i](v)| , |\shock[j](u)-\shock[j](v)| \}.
		\end{align}
		Let us assume that $\sum_{j\in\neighborhood[i]} \mean{\left|\gradient \unk\cdot \rr_{ij}\right|}_{ij} \neq 0$. (The other case is straightforward.) On one hand, for a non-degenerate FE mesh, we have that $c h \leq \rr_{ij} \leq C h$, $j \in \symneigh[i]$, for positive constants $c, C$ that do not depend on $h$. Using this fact in the definition of the shock detector \eqref{eq:alpha_lp}, we get:
		\begin{align}
			\detector[i](u)^{\frac{1}{q}} &= 
			\frac{\left|{\sum_{j\in\neighborhood[i]} \jump{\gradient \unk}_{ij}}\right|}{\sum_{j\in\neighborhood[i]} 2\mean{\left|\gradient \unk \cdot \rr_{ij}\right|}_{ij}}   =
			\frac{\left| \sum_{j\in\neighborhood[i]} \frac{u_i-u_j}{|\rr_{ij}|} +  \frac{u_i-u_j^\sym}{|\rr_{ij}^\sym|} \right|}
			{
				\sum_{j\in\neighborhood[i]} \frac{|u_i-u_j|}{|\rr_{ij}|} + \frac{|u_i-u_j^\sym|}{|\rr_{ij}^\sym|}
			}\label{alphabound} \\ & \eqsim
			\frac{\left| \sum_{j\in\neighborhood[i]} (u_i-u_j) +  (u_i-u_j^\sym) \right|}
			{
				\sum_{j\in\neighborhood[i]} {|u_i-u_j|} + {|u_i-u_j^\sym|}
			}
			. 
		\end{align}
		Now, we use the following result for two sequences $\{a_i\}_{i=1}^n$ $\{b\}_{i=1}^n$ (see \cite{barrenechea_2016} for further details):
		\begin{align}\label{absman}
			& \frac{|\sum_{i=1}^n a_i|}{\sum_{i=1}^n |a_i|}-
			\frac{|\sum_{i=1}^n b_i|}{\sum_{i=1}^n |b_i|}
			=
			\frac{|\sum_{i=1}^n a_i| - |\sum_{i=1}^n b_i|}{\sum_{i=1}^n |a_i|}
			+ \sum_{i=1}^n |b_i| \left( 
			\frac{1}{\sum_{i=1}^n |a_i|}-
			\frac{1}{\sum_{i=1}^n |b_i|}\right) \\ &
			\leq
			\frac{|\sum_{i=1}^n a_i- b_i|}{\sum_{i=1}^n |a_i|}
			+
			\frac{
				\sum_{i=1}^n |b_i|-\sum_{i=1}^n |a_i|
			}
			{\sum_{i=1}^n |a_i|}
			\leq 
			\frac{|\sum_{i=1}^n a_i- b_i| + \sum_{i=1}^n |a_i- b_i|}{\sum_{i=1}^n |a_i|} \\
			& \leq 2
			\frac{\sum_{i=1}^n |a_i- b_i|}{\sum_{i=1}^n |a_i|}. \label{absman}
		\end{align}
		Using simple algebraic manipulation, we have  $a^q-b^q = (a-b) \sum_{k=0}^{q-1} a^k b^{q-i}k$ for $q \in \mathbb{N}^+$. For $a,b \in [0,1]$, it leads to $|a^q-b^q| \leq q |a-b|$ (see \cite{barrenechea_2016}). This inequality, together with \eqref{alphabound} and \eqref{absman}, leads to:
		\begin{align}
			\frac{1}{q} | \shock[i](u) - \shock[i](v) | & \lesssim 
			\frac{\left| \sum_{j\in\neighborhood[i]} ((u-v)i-(u-v)_j) +  ((u-v)_i-(u-v)_j^\sym) \right|}
			{
				\sum_{j\in\neighborhood[i]} {|u_i-u_j|} + {|u_i- u_j^\sym|}
			}.\label{boundij}
		\end{align}
		On the other hand, the bounds
		$$|u_i-u_j | \leq
		\sum_{k\in\neighborhood[i]} | u_i-u_k | \quad \hbox{ and  } \quad |u_i-u_j | \leq 
		\sum_{k\in\neighborhood[j]} | u_j-u_k |,$$ 
		\eqref{viscobound}, and \eqref{boundij}, yield
		\begin{align}
			(\nu_\ij(u) - \nu_\ij(v)) (u_i - u_j)
			\lesssim &q h^{d-1} | \beta |
			\sum_{k\in\symneigh[i]} | (u-v)_i-(u-v)_k |  \label{visc2} \\ 
			& + q h^{d-1} | \beta |
			\sum_{k\in\symneigh[j]} | (u-v)_j-(u-v)_k |.
		\end{align}
		The second term is bounded by combining \eqref{boundij3}, \eqref{visc2}, and the fact that the number of elements surrounding a node is bounded above independently of $h$:
		\begin{align} 
			& \sum_{{i \in \nodes }} \sum_{{j\in\neighborhood[i]}} (\nu_\ij(u) -
			\nu_\ij(v)) \ell(i,j) u_j z_i \lesssim q h^{d-1} | \beta | |u-v|_\ell |z|_\ell.
		\end{align}
		Next, we have to prove that the nonlinear mass matrix is also Lipschitz continuous. First, we note that
		\begin{align}
			& \sum_{j\in\neighborhood[i]} 
			(1-\detector[i](u_h))(\shapef[j], \shapef[i]) u_j + \detector[i](u_h)(1,\shapef[i]) u_i
			\\ & = 
			\sum_{j\in\neighborhood[i]} (\shapef[j], \shapef[i]) u_j + 
			\detector[i](u_h) (\shapef[j], \shapef[i])  (u_i - u_j).
		\end{align}
		Thus
		\begin{align}
			\langle \M(u)u,z \rangle - \langle \M(v)v,z \rangle \leq  & 
			\sum_{{i \in \nodes }}
			\sum_{{j\in\neighborhood[i]}} (\shapef[i],\shapef[j])(u_j - v_j)z_i
			\\ &+ \sum_{{i \in \nodes }}\sum_{{j\in\neighborhood[i]}} (\shapef[i],\shapef[j])(u_i - u_j)(
			\detector[i](u_h) - \detector[i](v_h))z_i \\ &
			+ \sum_{{i \in \nodes }} \sum_{{j\in\neighborhood[i]}} (\shapef[i],\shapef[j])((u+v)_i - (u+v)_j)\detector[i](v_h) z_i. 
		\end{align}
		Bounds for the second and third term follow the same lines as above. For the second term, we proceed as in \eqref{boundij3}, getting:
		\begin{align}
			& \sum_{{i \in \nodes }} \sum_{{j\in\neighborhood[i]}} 
			(\shapef[i],\shapef[j])(u_i - u_j)
			(\detector[i](u_h) - \detector[i](v_h))
			z_i \\ &
			\lesssim
			\left | \sum_{{i \in \nodes }} \half \sum_{{j\in\neighborhood[i]}} (\shapef[i],\shapef[j]) (\detector[i](u_h) - \detector[i](v_h))^2 (u_i - u_j)^2 \right |^\half \times \| z\| \\ &
			\lesssim q h^{\frac{d}{2}} |u-v|_\ell \|z\|.
		\end{align}
		where we have used the spectral equivalence of the consistent and lumped mass matrices in the last inequality. The first and third term are easily bounded as
		\begin{align}
			&\sum_{{i \in \nodes }}
			\sum_{{j\in\neighborhood[i]}} (\shapef[i],\shapef[j])(u_j - v_j)z_i  \leq \|u-v\|\|z\|, \\
			&\sum_{{i \in \nodes }} \sum_{{j\in\neighborhood[i]}} (\shapef[i],\shapef[j])((u+v)_i - (u+v)_j)\detector[i](v_h) z_i \leq q h^{\frac{d}{2}} |u-v|_\ell \|z\|. 
		\end{align}
		It proves the theorem.

\bibliographystyle{siam}
\bibliography{bib-cf}

\end{document}